\documentclass[a4paper,10pt]{article}
\usepackage{authblk}

\usepackage[english]{babel}
\usepackage[latin1]{inputenc}
\usepackage{csquotes}
\usepackage{amsfonts,amsmath,amsthm}
\usepackage{empheq}
\usepackage{dsfont}
\usepackage{cancel}
\usepackage[titletoc,title]{appendix}
\usepackage[backend=bibtex8,doi=false,eprint=false,giveninits=true,isbn=false,style=numeric-comp,url=false,maxnames=99]{biblatex}
\makeatletter
\def\blx@maxline{77}
\makeatother
\DeclareRedundantLanguages{english,german,french}{english,german,ngerman,french}

\usepackage{cases}
\usepackage{mathabx}
\usepackage{xfrac}
\usepackage{fancyhdr}
\usepackage{color}
\usepackage[colorlinks]{hyperref}
\definecolor{blue75}{rgb}{0,0,.75}
\definecolor{green75}{rgb}{0,.75,0}

\newcommand{\cb}{\color{blue}}
\newcommand{\cred}{\color{red}}

\hypersetup{colorlinks=true, urlcolor=blue75,linkcolor=blue75,citecolor=green75,pdfstartview=FitB,bookmarksopen=true,bookmarksopenlevel=1}
\usepackage[a4paper, left=2.5cm, right=2.5cm, top=2.5cm,bottom=2cm]{geometry}
\usepackage{constants}
\newconstantfamily{C}{
symbol=C,
format=\parenthezises,
reset={section}
}
\usepackage{enumerate}

\usepackage{graphicx}
\graphicspath{ {images/} }
\usepackage{wrapfig}
\usepackage{figbib}
\usepackage{changes}
\allowdisplaybreaks
\begin{document}
\newcommand{\R}{\mathbb{R}}
\newcommand{\N}{\mathbb{N}}
\newcommand{\Ss}{\mathbb{S}}

\newcommand{\calL}{\mathcal{L}}
\newcommand{\calI}{\mathcal{I}}
\newcommand{\Ker}{\text{Ker }}
\newcommand{\spn}{\text{span }}

\newcommand{\port}{p^\perp}

\newcommand{\ua}{u^{\alpha}}
\newcommand{\ub}{u^{\beta}}
\newcommand{\uaq}{u^{q-1+\alpha}}
\newcommand{\Tm}{T_{max}}
\newcommand{\oa}{\bar{\Omega}}
\newcommand{\bu}{\bar{u}}
\newcommand{\ot}{\Omega \times (0,T)}
\newcommand{\uj}{u_{{\cred l}k}}
\newcommand{\ujk}{u_{{\cred l}m_{\cred o}}}
\newcommand{\hj}{h_{{\cred l}k}}
\newcommand{\hjk}{h_{{\cred l}km_{\cred o}}}
\newcommand{\ck}{C^{\kappa,\frac{\kappa}{2}}(\oa \times [0,T])}
\newcommand{\io}{\int_{\Omega}}
\newcommand{\tr}{\text{tr}}
\newcommand{\D}{\mathbb{D}}
\newcommand{\dij}{d_{ij}}
\newcommand{\td}{\text{d}}
\newcommand{\bh}{\bar{h}}

\newtheorem{Theorem}{Theorem}[section]
\newtheorem{Assumptions}[Theorem]{Assumptions}
\newtheorem{Corollary}[Theorem]{corollary}
\newtheorem{Convention}[Theorem]{convention}
\newtheorem{Definition}[Theorem]{Definition}

\newtheorem{Lemma}[Theorem]{Lemma}
\newtheorem{Notation}[Theorem]{Notation}
\theoremstyle{definition}
\newtheorem{Example}[Theorem]{Example}
\newtheorem{Remark}[Theorem]{Remark}
\numberwithin{equation}{section}
\title{On a mathematical model for cancer invasion with repellent pH-taxis and nonlocal intraspecific interaction}
\author{Maria Eckardt\thanks{\href{mailto:eckardt@mathematik.uni-kl.de}{eckardt@mathematik.uni-kl.de}} \  and Christina Surulescu\thanks{\href{mailto:surulescu@mathematik.uni-kl.de}{surulescu@mathematik.uni-kl.de}}\\
	Felix-Klein-Zentrum für Mathematik, RPTU Kaiserslautern-Landau,\\ Paul-Ehrlich-Str. 31, 67663 Kaiserslautern, Germany
}

\date{\today}
\maketitle

\begin{abstract} 
	\noindent
	Starting from a mesoscopic description of cell migration and intraspecific interactions we obtain by upscaling an effective  reaction-difusion-taxis equation for the cell population density involving spatial nonlocalities in the source term and biasing its motility and growth behavior according to environmental acidity. We prove global existence, uniqueness, and boundedness of a nonnegative solution to a simplified version of the coupled system describing cell and acidity dynamics. A 1D study of pattern formation is performed. Numerical simulations illustrate the qualitative behavior of solutions.

\end{abstract}

\section{Introduction}

Migration, proliferation, and differentiation of cells are influenced by biochemical and biophysical characteristics of their surroundings, which they perceive by way of transmembrane units like ion channels, receptors, etc. Increasing experimental evidence suggests that cells are able to sense such cues not only where they are, but also at larger distances, up to several cell diameters around their current position \cite{GM17,Kornberg2014,SSM17}. This led to mathematical models accounting for various types of nonlocalities, most of them addressing cell-cell and/or cell-matrix adhesions; we refer to the review article \cite{CPSZ} and references therein. The settings typically involve reaction, diffusion and drift terms, whereby the latter contain an integral operator to characterize the so-called adhesion velocity over the interaction range. In \cite{Eckardt2020} was performed a rigorous passage from a  cell-matrix adhesion model to a reaction-diffusion-haptotaxis equation when the sensing radius is becoming infinitesimally small, thus recovering the local PDE formulation from that featuring the mentioned nonlocality. The remote sensing of signals by cells affects, however, not only motility, but also proliferation, growth, and phenotypic switch, either directly - by occupancy of transmembrane units on cellular extensions like cytonemes and folopodia and subsequently initiated signaling pathways, or in an indirect manner - as effects of altered migratory and aggregation behavior. Models involving reaction-diffusion equations with nonlocal source terms have been proposed in various contexts, including biological and ecological ones, see e.g. \cite{kav-suzuki,vol2} and references therein for rather generic settings, \cite{BCL,Bian2018b,Ren2022} for chemotaxis systems, and \cite{Negreanu2013,SMLC,SSMG} for equations dedicated to tumor growth. We refer to \cite{CPSZ,kav-suzuki,vol2} for some reviews of model classes addressing this type of nonlocality.\\[-2ex]

\noindent
As far as growth and migration of cell populations are concerned, the reaction-diffusion models with nonlocal source terms 
\begin{equation}
u_t=\nabla \cdot (D\nabla u)+F(u)
\end{equation}
typically feature $F(u)=\mu J*u(1-u)$ to describe nonlocal stimulation of growth (see e.g. \cite{SSMG,vol2}), or $F(u)=\mu u^\alpha (1-J*u^\beta)$, which characterizes competition between (bunches of) cells for available resources in their surroundings, attempting, e.g., to prevent overcrowding. In the context of (tumor) cell migration such models have been handled e.g. in \cite{SMLC}, where intra- and interspecific nonlocal interactions led to an ODE-PDE system for the interplay between cancer cells peforming linear diffusion and haptotaxis with the extracellular matrix being (nonlocally) degraded by the cells and remodeled with the mentioned growth limitation. We also refer to \cite{CPSZ,LCS} for short reviews of models with source terms of this type and therewith associated mathematical challenges.\\[-2ex] 

\noindent
In this note we propose and analyze a model for tumor cell migration involving myopic diffusion, repellent pH-taxis, and a nonlocal source term of the competition type mentioned above. The cross-diffusion system is obtained upon starting from the mesoscopic description of cell migration via a kinetic transport equation for the space-time distribution function of cells sharing some velocity regime. An appropriate upscaling relying on diffusion dominance then leads to the effective macroscopic equation for the cancer cell density, with precisely specified diffusion and drift coefficients. The remaining of this paper is structured as follows: Section \ref{sec:model} contains the model deduction with the mentioned upscaling. Section \ref{sec:analysis} is dedicated to the mathematical analysis of the obtained nonlocal macroscopic system, in terms of global existence, uniqueness, and boundedness of a solution to a simplified version of the problem. In section \ref{sec:asymptotic} we study the asymptotic behavior. Section \ref{sec:pattern} offers a 1D study of pattern formation for the equations handled in Section \ref{sec:analysis}, but only involving constant motility coefficients. In Section \ref{sec:numerics} we provide numerical simulations to illustrate the qualitative behavior of solutions to the investigated nonlocal problem. Section \ref{sec:discussion} contains a discussion of the results.

\section{Modeling}\label{sec:model}

In this section we start from a mesoscopic description of cell migration and intrapopulation interactions and deduce (in a non-rigorous way) effective equations on the macroscopic scale of cell population dynamics. The deduction closely follows that in \cite{LCS}, however extends it, by accounting here for the repellent effects of acidity eventually leading on the population scale to chemorepellent pH-taxis.\\[-2ex]

\noindent
Tumor migration and spread are typically assessed on the macroscopic scale of the cancer cell population via biomedical imaging. The involved processes are, however, highly complex and originate at the lower levels of cell aggregates sharing -beside time-space dynamics- one or several further traits (e.g., velocity, phenotypic state or other so called 'activity variables'), down to microscopic events on individual cells. This multiscale character of cell migration can be captured (at least partially) by models within the kinetic theory of active particles (KTAP) framework formulated by Bellomo et al. (see e.g., \cite{bellomo2008,BBGO17} and references therein). Starting from kinetic transport equations (KTEs), a large variety of (spatially) local and nonlocal models have been proposed and various kinds of upscaling and moment closure methods have been performed in order to deduce their macroscopic limits which enable a mathematically more efficient handling, see e.g. \cite{Corbin2018,CHP07,CEKNSSW,EHKS14,EHS,hillenM5,Conte2023,Conte_Surulescu,Kumar2021,Kumar2020,Chalub2004,dietrich2020,Engwer2015,ZSnovel21,Loy2019,conte-loy22}. The obtained macroscopic equations carry in the coefficients of their motility and source terms some of the traits from the mesoscale on which KTEs were formulated. Those coefficients are no longer 'guessed' as in the case of stating  reaction-diffusion-taxis directly on the population level and the diffusion is often of the 'myopic' type, involving a drift correction. We will perform here a diffusion-dominated upscaling of mesoscale dynamics. 

\noindent
We will use the following notations:
\begin{itemize}
	\item $p=p(t,x,v)$: distribution function of cells having at time $t$ and position $x\in \R^n$ the velocity $v\in V$;
	\item $V=[s_1,s_2]\times \Ss^{n-1}$: velocity space. Thereby, $s_1,s_2$ denote the minimum, respectively the maximum speed of a cell,  $\theta \in \Ss^{n-1}$ represents the cell direction;
	\item $u(t,x)=\int _Vp(t,x,v)\ dv$: macroscopic cell density;
	\item $h(t,x)$: concentration of protons. This is a macroscopic quantity throughout this note.
\end{itemize}

\noindent
The kinetic transport equation (KTE)
\begin{equation}\label{eq:KTE1}
p_t+v\cdot \nabla_xp=\calL [h]p+\tilde \mu \calI[p,p]
\end{equation}
characterizes the mesoscopic dynamics of the considered cell population. This is the framework set in \cite{ODA}, which assumes that changes in $p$ are due to velocity jumps accompanied by reorientations dictated by a turning kernel contained in the operator $\calL[h]$. \\[-2ex]

\noindent
The first term on the right hand side of \eqref{eq:KTE1} represents the so-called turning operator. The second term describes growth/decay of cells due to intraspecific proliferative/competitive interactions, while $\tilde \mu >0$ is the constant interaction rate.\footnote{We could actually consider $\tilde \mu $ to be a function of $x$ and/or $t$ (but not of derivatives w.r.t. these variables) and even of $h$. The latter would allow us to account e.g., for the unfavorable effect of acidity on the proliferation of tumor cells. The deduction done here works then exactly in the same way. In fact, our analysis in Section \ref{sec:analysis} is performed in the case where such $h$-dependence is considered.} With a small constant $\epsilon >0$ relating to the cell size and to the distance at which cells can sense signals in their proximity, we will assume that $\tilde \mu =\epsilon ^2\mu $. This means that cells have a much higher preference to motility (in particular, to changing direction) than to interaction and crowding. \\[-2ex]

\noindent
We assume that the turning operator is of the form
\begin{equation}\label{eq:calL-0}
	\calL[h](p)=\int _V\Big (T[h](v,v')p(t,x,v')-T[h](v',v)p(t,x,v)\Big )\ dv',
\end{equation}
with the turning rate $T[h](v,v')\ge 0$ chosen such that the reorientation is a Poisson process with rate
$$\lambda [h]=\int _VT[h](v,v')dv,$$
hence such that $T[h]/\lambda [h]$ is a kernel giving the probability density for a change of the velocity regime of a cell from $v'$ to $v$. In particular, this means that $\calL[h]$ is preserving mass. The reorientation of cells depends on the acidity of their environment (expressed by the concentration $h$ of protons).\\[-2ex]

\noindent
In the following we assume that the turning rate has an asymptotic expansion of the form
\begin{equation}\label{eq:turning-rate-exp}
T[h]=T_0[h]+\epsilon T_1[h]+O(\epsilon ^2),
\end{equation}	
thus the turning operator admits itself an expansion
\begin{equation}\label{eq:calL}
	\calL[h](p):=L_0[h](p)+\epsilon L_1[h](p)+O(\epsilon^2),
\end{equation}
where $L_0[h]$ and $L_1[h]$ are linear operators,
\begin{align}\label{eq:Ls}
L_i[h](p)(t,x,v)&=\int _V[T_i[h](v,v')p(t,x,v')-T_i[h](v',v)p(t,x,v)]\ dv',\qquad i=0,1.
\end{align}	

\noindent
For $\calI$ we consider as in \cite{LCS} the form
\begin{equation}\label{eq:calI}
\calI[p,p](t,x,v)=\frac{p^\alpha(t,x,v)}{\int _VM^\alpha(x,v)\ dv}-\frac{1}{\int _VM^{\alpha+\beta }(x,v)}p^\alpha (t,x,v)\int _\Omega J(x,x')	p^\beta (t,x',v)\ dx',
\end{equation}
where: $\alpha,\beta>0$ are constants, $J(x,x')$ is a function weighting the interactions between (bunches of) cells sharing the same velocity regime within a bounded domain $\Omega\subset \R^n$. We assume that $J$ depends on the distance between interacting (clusters of) cells and take $J(x,x')=J(x-x')$, also requiring $J$ to satisfy
\begin{align}\label{eq:cond-J}
& \int _VJ(x)\ dx=1\\
& \inf \limits_{B_{\text{diam}(\Omega)}(0)}J\ge \eta \quad \text{for some }\eta >0.
\end{align}
We also assume that there exists a bounded velocity distribution $M(x,v)>0$ such that:
\begin{enumerate}
	\item $\int _VM(x,v)\ dv=1$, i.e. $M$ is a kernel w.r.t. $v$.
	\item $\int _VvM(x,v)\ dv=0$, i.e. the flow produced by the equilibrium distribution $M(v)$ vanishes.
	\item The rate $T_0[h](v,v')$ satisfies the detailed balance equation
	$$T_0[h](v,v')M(v')=T_0[h](v',v)M(v).$$
	\item The turning rate $T_0[h](v,v')$ is bounded and there exists $\sigma >0$ such that
	$$T_0[h](v,v')\ge \sigma M(x,v),\quad \text{for all }(v,v')\in V\times V,\ x\in \R^n,\ t>0.$$
	\end{enumerate}

\noindent
The following lemma summarizing the properties of the operator $-L_0$ can be easily verified (see e.g. \cite{BB15,Chalub2004}).

\begin{Lemma}\label{lemma-BB}
Let $L_0[h]$ be the operator defined in \eqref{eq:Ls}. Then $-L_0[h]$ has the following properties:
\begin{itemize}
	\item [(i)] $-L_0[h]$ is positive definite w.r.t. the scalar product and the associated norm in the weighted space $L^2(V,\frac{dv}{M(x,v)})$, and self-adjoint: for all $p,\zeta \in L^2(V,\frac{dv}{M(x,v)})$ it holds that
	$$\int _VL_0[h](p)(v)\frac{\zeta(v)}{M(v)}\ dv=\int _VL_0[h](\zeta)(v)\frac{p(v)}{M(v)}\ dv.$$
	\item[(ii)] For $\phi \in L^2(V,\frac{dv}{M(x,v)})$, the equation $L_0[h](\zeta ) =\phi$ has a unique solution $\zeta \in L^2(V,\frac{dv}{M(x,v)})$ satisfying\footnote{Here and in the remaining of this section we use the notation $\bar \zeta :=\int _V\zeta (v)\ dv$ for any $V$-integrable function $\zeta$ (hence also $u=\bar p$).}
	$\bar \zeta =0$ iff $\bar \phi =0$. 
	\item[(iii)] $\Ker L_0[h]=\spn (M(v))$.
	\item[(iv)] The equation $L_0[h]( \psi) =vM(v)$ has a unique solution $\psi(v)=:L_0[h]^{-1}(vM(v))$ (this is actually a pseudoinverse).
\end{itemize}
\end{Lemma}

\begin{Example}\label{bsp} Consider $T_0[h](v,v'):=\lambda _0[h]M(v)$, with $\lambda[h]\ge \lambda_0[h]>0$ for any $h$. This obviously satisfies the properties 3. and 4. in our above assumption. With this choice,
\begin{equation}\label{eq:partic-L_0}
L_0[h](p)=\lambda_0[h](M(v)u-p)
\end{equation}
and it is straightforward to see that this operator satisfies the properties in Lemma \ref{lemma-BB} and the function $\psi $ in (iv) becomes $\psi(v)=-vM(v)/\lambda_0[h]$ if $\psi \in (\spn (M(v)))^\perp$.\footnote{This is actually the case even if $T_0$ has a more general form (depending only on $v$ and not on $v'$) without having to satisfy condition 2. }
\end{Example}

\noindent
Equation \eqref{eq:KTE1} is supplemented with the macroscopic PDE for proton concentration:
\begin{equation}\label{eq:h}
h_t=D_H\Delta h+g(u,h),
\end{equation}
where $D_H>0$ is the diffusion constant and $g(u,h)$ represents production by tumor cells and uptake (e.g., by blood capillaries - not explicitly modeled in this note) or decay.

\noindent
We also consider initial conditions for $p$ and $h$:
\begin{align}\label{eq:ICs_p-h}
p(0,x,v)=p_0(x,v),\qquad h(0,x)=h_0(x), \quad x\in \Omega \subseteq \R^n,\ v\in V.	
\end{align}

\noindent
Together with these, equations \eqref{eq:KTE1},\eqref{eq:h} form a meso-macro system describing the dynamics of the (mesoscopic) cell distribution in response to acidity in the extracellular space.\\[-2ex]

\noindent
We perform a parabolic scaling to obtain the diffusion limit of the KTE \eqref{eq:KTE1}. This means that we rescale the time and space variables as follows: 
$$\hat t=\epsilon^2t,\quad \hat x=\epsilon x.$$
Subsequently we will drop the '\ $\hat {}$\ ' symbol and the $\epsilon $-dependency of the solution $p^\epsilon$ to the resulting KTE, in oder not to complicate the writing. Then, \eqref{eq:KTE1} becomes
\begin{equation}\label{eq:KTE2}
	\epsilon p_t+v\cdot \nabla_xp=\frac{1}{\epsilon}\calL[h]p+\mu \epsilon\calI[p,p].
\end{equation}

\noindent
Now consider the decomposition (Chapman-Enskog expansion)
\begin{equation}\label{eq:Chapman-Enskog}
p(t,x,v)=F(u)(t,x,v)+\epsilon \port (t,x,v),	
\end{equation}
with $\int _V\port (t,x,v)\ dv=0$, thus $\port \in (\spn(M(v)))^\perp$, and $F(u)\in \spn(M(v))$ such that $\int _VF(u)\ dv=u$. A natural choice is $F(u)(t,x,v):=M(x,v)u(t,x)$, which we will subsequently adopt.\\[-2ex]

\noindent
Then observe that $$\calI[p,p]=\calI[M(v)u+\epsilon \port,M(v)u+\epsilon \port]=\calI[M(v)u,M(v)u]+O(\epsilon)$$ and 
\eqref{eq:KTE2} becomes
\begin{align}\label{eq:KTE3}
	\partial_t(M(v)u)+\epsilon \partial_t\port +\frac{1}{\epsilon}v\cdot \nabla_x(M(v)u)+v\cdot \nabla_x\port =\frac{1}{\epsilon}L_0[h](\port)&+\frac{1}{\epsilon}L_1[h](M(v)u)+ L_1[h](\port) \notag \\
	&+\mu \calI[M(v)u,M(v)u]+O(\epsilon).
\end{align}

\noindent
Let $P:L^2(V,\frac{dv}{M(v)})\to \Ker L_0[h]$ be the projection operator. Then
$$P(\phi)=M(v)\bar \phi,\qquad  \phi \in L^2(V,\frac{dv}{M(v)}).$$

\noindent
It is easy to verify that the following lemma holds (see, e.g., \cite{BB15}).

\begin{Lemma}\label{lem:properties-of-projection}
The projection operator $P$ has the following properties:
\begin{itemize}
	\item [(i)] $(I-P)(M(v)u)=P(\port)=0$.
	\item[(ii)] $(I-P)(v\cdot \nabla_x(M(v)u))=v\cdot \nabla_x(M(v)u)$.
	\item[(iii)] $(I-P)(L_0[h](M(v)u))=L_0[h](M(v)u)$ and $(I-P)(L_1[h](M(v)u))=L_1[h](M(v)u)$.
	\item[(iv)] $(I-P)(L_1[h](\port))=L_1[h](\port)$.
\end{itemize}	
\end{Lemma}

\noindent
If we now apply $I-P$ to \eqref{eq:KTE3} we get
\begin{align}\label{eq:KTE4}
\epsilon \partial_t\port +\frac{1}{\epsilon}v\cdot \nabla_x(Mu)+(I-P)(v\cdot \nabla_x\port)=\frac{1}{\epsilon}L_0[h](\port)&+\frac{1}{\epsilon}L_1[h](Mu)+L_1[h](\port) \notag \\
&+\mu \calI[Mu,Mu]+O(\epsilon).	
\end{align}

\noindent
Integrating \eqref{eq:KTE3} w.r.t. $v$ gives (at leading order) the macroscopic PDE\footnote{involving nonlocalities w.r.t. velocity}
\begin{equation}\label{eq:KTE5}
u_t+\int_Vv\cdot \nabla _x\port \ dv=\mu \int _V\calI[Mu,Mu]\ dv.
\end{equation}

\noindent
On the other hand, from \eqref{eq:KTE4} we obtain (again at leading order)
\begin{equation}\label{eq:KTE6}
L_0[h](\port)=v\cdot \nabla_x(Mu)-L_1[h](Mu).
\end{equation}

\noindent
Since $\int _VL_1[h](Mu)\ dv=0$, we see that the integral w.r.t. $v$ of the right hand side in \eqref{eq:KTE6} vanishes, so we can pseudo-invert $L_0[h]$ to obtain 
\begin{equation}\label{eq:p-orthog}
\port =L_0[h]^{-1}\Big (v\cdot \nabla_x(Mu)- L_1[h](Mu) \Big ).	
\end{equation}	

\noindent
Plugging this into 	\eqref{eq:KTE5} gives

\begin{equation}\label{eq:fast-macro}
u_t+\int _Vv\cdot \nabla_x\Big (L_0[h]^{-1}(v\cdot\nabla_x(Mu))- L_0[h]^{-1}(L_1[h](Mu)) \Big )	=\mu \int_V\calI[Mu,Mu]\ dv.
\end{equation}	

\noindent
For the right hand side in \eqref{eq:fast-macro} we have
$$\mu \int_V\calI[Mu,Mu]\ dv=\mu u^\alpha (1-J*u^\beta).$$

\noindent
For the first transport term on the left hand side we compute
\begin{align*}
\int _Vv\cdot \nabla _x	\Big (L_0[h]^{-1}(v\cdot\nabla_x(Mu))\Big )\ dv&=\nabla _x\cdot \Big (\frac{1}{\lambda_0[h]}\nabla _x \cdot \Big (\int _Vv\otimes vM(v)\ dv\ u\Big )\Big )\\
&=-\nabla _x\cdot \Big (\frac{1}{\lambda_0[h]}\nabla _x \cdot  (\mathbb D\ u )\Big ),
\end{align*}
where we applied the observations made at the end of Example \ref{bsp} and denoted by 
$$\mathbb D(x):=\int _Vv\otimes vM(x,v)\ dv$$
the diffusion tensor of tumor cells. \\[-2ex]

\noindent
For the second transport term on the left hand side of \eqref{eq:fast-macro} we have
\begin{align*}
-\int _Vv\cdot \nabla_x\Big (L_0[h]^{-1}(L_1[h](M(v)u))\Big )\ dv&=- \nabla_x\cdot \int _Vv	L_0[h]^{-1}(L_1[h](M(v)u))\ dv\\
&=-\nabla_x\cdot \int _VvM(v) \frac{1}{M(v)}L_0[h]^{-1}(L_1[h](M(v)u))\ dv\\
&=-\nabla_x\cdot \int _VL_0[h](\psi(v))\frac{1}{M(v)}L_0[h]^{-1}(L_1[h](M(v)u))\ dv\\
&=- \nabla_x\cdot \Big (\int _V\frac{\psi(v)}{M(v)}L_1[h](M(v))\ dv\ u\Big )\\
&= \nabla_x\cdot \Big (u\Gamma [h]\Big ),
\end{align*}	
where we used the fact that $L_0$ is self-adjoint, $\psi(v)=-vM(v)$ is its pseudo-inverse, and the notation
$$\Gamma [h](x):=\frac{1}{\lambda_0[h]}\int _VvL_1[h](M(x,v))\ dv.$$

\noindent
With the above calculations \eqref{eq:fast-macro} becomes
\begin{equation}\label{eq:echt-macro}
u_t	
-\nabla _x\cdot \Big (\frac{1}{\lambda_0[h]}\nabla _x \cdot  (\mathbb D\ u )\Big )
+\nabla_x\cdot \Big (u\Gamma [h]\Big )=\mu u^\alpha (1-J*u^\beta).
\end{equation}

\noindent
To specify $\Gamma[h]$ we consider\footnote{a similar choice has been proposed in \cite{Chalub2004}} 
$T_1[h](v,v'):=-a(h)v\cdot \nabla h+b(h)v'\cdot \nabla h$ with $a,b\ge0$. Then we compute
\begin{align*}
\int _V	vL_1[h](M(x,v))\ dv=-a(h)\frac{s_2^{n+2}-s_1^{n+2}}{n(n+2)}|\Ss^{n-1}|\ \mathbb I \nabla h-\frac{b(h)}{|V|}\mathbb D\nabla h,
\end{align*}
recalling that $V=[s_1,s_2]\times \Ss^{n-1}$, thus $|V|=\frac{s_2^n-s_1^n}{n}|\Ss^{n-1}|$. With the notation $\mathbb T(x):=a(h)\frac{s_2^{n+2}-s_1^{n+2}}{n(n+2)}|\Ss^{n-1}|\ \mathbb I+\frac{b(h)}{|V|}\mathbb D$ we obtain 
\begin{equation*}
	\Gamma [h](x)=-\frac{1}{\lambda_0[h]}\mathbb T(x)\nabla h,
\end{equation*}
which leads to the macroscopic PDE
\begin{equation}\label{eq:macro-PDE-1}
u_t=\nabla _x\cdot \Big (\frac{1}{\lambda_0[h]}\nabla _x \cdot  (\mathbb D(x) u )\Big )+\nabla _x\cdot \Big (\frac{u}{\lambda_0[h]}\mathbb T(x)\nabla h \Big )+\mu u^\alpha (1-J*u^\beta).
\end{equation}
The particular choice $\lambda_0[h]:=1$, $a(h):=0$, $b(h):=|V|$ leads to the first equation in \eqref{IBVP}.\\[-2ex]

\noindent
The first term on the right hand side of \eqref{eq:macro-PDE-1} represents (myopic) diffusion, the second one characterizes repellent chemotaxis, away from increasing gradients of proton concentration\footnote{as in \cite{Conte_Surulescu,CEKNSSW,Kolbe2021,Kumar2021,Kumar2020} we call this a repellent pH-taxis}, while the last is a source term accounting for tumor cell growth enhanced or limited by intraspecific interactions.\\[-2ex]

\noindent
The above deduction of a macroscopic reaction-diffusion-taxis is merely formal; the nonlinear source term prevents applying the proof of the rigorous derivation from \cite{Chalub2004}. The following section will be dedicated to proving global existence and boundedness of nonnegative solutions to the coupled PDE system for $u$ and $h$ obtained on the macrolevel by considering the above much simplified forms of the coefficient functions $\lambda_0,a, b$. The previous calculations were made for $x\in \R^n$, however we can restrict to a bounded domain $\Omega \subset \R^n$ upon proceeding as in \cite{CEKNSSW,dietrich2020,Plaza} and assuming no flux of cells or protons through the boundary.

\section{Mathematical analysis}\label{sec:analysis}

\noindent
Let $\Omega \subset \R^n$ be a bounded domain with smooth enough boundary and outer unit normal $\nu$. We consider the model
\begin{align} \label{IBVP}
\begin{cases}
u_t = \nabla\nabla:(\D(x)u) + \nabla \cdot (\D(x)u\nabla h) + \mu(h) \ua (1- J \ast \ub), & x\in \Omega, \, t > 0,\\
h_t = D_H \Delta h + g(u,h), & x \in \Omega, \, t > 0,\\
(\D(x)\nabla u  + \nabla \cdot \D(x)u + \D (x) u \nabla h) \cdot \nu = \nabla h \cdot \nu = 0, & x \in \partial \Omega, \, t>0,\\
u(x,0) = u_0(x), \, h(x,0) = h_0(x), & x \in \Omega,
\end{cases}
\end{align}
where $u$ denotes the cell density and $h$ the acid concentration. Here, the convolution over $\Omega$ is as usually given by 
\begin{align*}
J \ast \ub (x,t) = \int_{\Omega} J(x-y) \ub(y,t) \, \td y.
\end{align*}
For our diffusion tensor $\D = \left(d_{ij}\right)_{i,j=1,...,n}$, we assume that $\dij \in C^{1}(\oa)$. 
Moreover, $\D$ satisfies the uniform parabolicity and boundedness condition, i.e. there are $ B_1,\,  B_2 >0$ such that for all $\xi \in \R^n$ and $x \in \oa$ it holds that
\begin{align} \label{upc}
B_1 |\xi|^2 \leq \sum_{i, j =1}^n \dij(x) \xi_j \xi_i \leq  B_2 |\xi|^2.
\end{align} 
Additionally, we assume that for $x \in \partial \Omega$ and $\xi \in \R^n$ with $\xi \cdot \nu = 0$ on $\partial \Omega$ and $|\xi| \neq 0$ it holds that
\begin{align} \label{LS}
4 \sum_{i,j=1}^n \dij \xi_i \xi_k \sum_{k,l = 1}^n d_{kl} \nu_k \nu_l - \left(\sum_{i,j=1}^n \dij (\xi_i \nu_j + \xi_j \nu_i)\right)^2 >0.
\end{align}
Condition \eqref{LS} is for example satisfied if $\D$ is a multiple of the identity.\\[-2ex]

\noindent
The exponents $\alpha, \beta \geq 1$ satisfy (as in \cite{LCS}) 
\begin{align}\label{cond:alphabeta}
\alpha < \begin{cases} 1+ \beta, &\quad n= 1,2,\\ 1+ \frac{2\beta}{n}, &\quad n >2.\end{cases}
\end{align}
On the remaining functions and parameters we make the subsequent assumptions:
\begin{itemize}
\item $u_0 \in C(\oa)$ and $u_0 \geq 0$,
\item $h_0 \in W^{1,\infty}(\Omega)$ and $0 \leq h_0 \leq H$, $h_0 \not\equiv H$, where $H$ is a positive constant,
\item $\mu$ is Lipschitz-continuous with constant $L_{\mu}$, satisfying $0 \leq \mu$ and $\mu(h) \geq \delta >0$ for $h \leq H$,
\item $g \in C^1(\R_0^+ \times \R_0^+)$ with $\nabla g \in (L^{\infty}(\R_0^+ \times \R_0^+))^2$, $0 \leq g(u,0) \leq G$ and $g(u, H) \leq 0$ for $u \in \R_0^+$,
\item $J \in L^p(B_{\text{diam}(\Omega)}(0))$ for some $p \in (1, \infty)$, $0 < \eta \leq J$,
\item $D_H > 0$.
\end{itemize}
\noindent
By convention the term $K_i>0$ denotes a positive constant for all $i \in \N$ (or, respectively, a positive function of its arguments).


\subsection{Local existence in an approximate problem}
The Stone-Weierstraß theorem implies that there is a sequence of diffusion tensors $(\D_l)_{l \in \N}$ with $\D_l = \left(d_{lij}\right)_{i,j=1,...,n}$ s.t. $d_{lij} \in C^{2+\vartheta}(\oa)$ for $\vartheta \in (0,1)$ and $\D_l \rightarrow \D$ in $C^1(\oa)^{n\times n}$ for $l\to \infty$.  
Moreover, $\D_l$ satisfies \eqref{LS} and the uniform parabolicity condition for all $l \in \N$, i.e. there are $0 < D_1 < B_1 < B_2 < D_2$ such that for all $\xi \in \R^n$, $x \in \oa$ and $l \in \N$ it holds that
\begin{align} \label{upc}
D_1 |\xi|^2 \leq \sum_{i, j =1}^n d_{lij}(x) \xi_j \xi_i \leq D_2 |\xi|^2.
\end{align} 

\noindent
For $l \in \N$ we consider the approximate problem 
\begin{align} \label{IBVPL}
\begin{cases}
\partial_tu_l = \nabla\nabla:(\D_l(x)u_l) + \nabla \cdot (\D_l(x)u_l\nabla h_l) + \mu(h_l) \ua_l (1- J \ast \ub_l), & x\in \Omega, \, t > 0,\\
\partial_th_l = D_H \Delta h_l + g(u_l,h_l), & x \in \Omega, \, t > 0,\\
(\D_l(x)\nabla u_l + \nabla \cdot \D_l(x)u_l + \D_l (x) u_l \nabla h_l) \cdot \nu = \nabla h_l \cdot \nu = 0, & x \in \partial \Omega, \, t>0,\\
u_l(x,0) = u_0(x), \, h(x,0) = h_0(x), & x \in \Omega.
\end{cases}
\end{align}

\begin{Lemma} \label{wsollem}
For all $l \in \N$ there are $\Tm>0$ and a weak solution $(u_l,h_l)$ of \eqref{IBVPL} such that for all $T \in (0,\Tm)$ it holds that $u_l \in C(\oa \times [0,T])) \cap L^2(0,T;H^1(\Omega))$ and\footnote{$W^{k,l}_p(\ot )$ denotes the Sobolev space of functions having weak derivatives in $L^p(\ot)$, namely up to order $k$ w.r.t. space and up to order $l$ w.r.t. time} $h_l \in C(\oa \times [0,T]) \cap L^{\infty}(0,T;W^{1,{\infty}}(\Omega)) \cap W^{2,1}_{2}(\ot)$ and $(u,h)$ satisfies for a.e. $t \in (0,T)$ and all $\eta \in W^{1,1}_2(\ot)$ it holds that
\begin{align} 
&\int_{\Omega} u_l(x,t) \eta(x,t) \,dx - \int_0^t \int_{\Omega}  u_l \eta_t \,dx\,ds +  \int_0^t \int_{\Omega} \left( \nabla \cdot \D_l u_l + \D_l \nabla u_l + \D_l u_l\nabla h_l\right) \cdot \nabla \eta  \,dx\,ds \nonumber \\
&= \int_0^t \int_{\Omega} \mu(h_l)\ua_l (1- J \ast \ub_l)\eta \,dx\,ds + \int_{\Omega} u_0(x) \eta(x,0) \, dx, \label{wsu}\\
&u_l(0) = u_0 \text{ in } L^2(\Omega)  \label{uaw}
\end{align}
and
\begin{align} 
\partial_th_l &= D_H \Delta h_l + g(u_l,h_l) \quad &\text{ a.e. in } \Omega \times (0,\Tm) \label{wsh}\\
\nabla h_l \cdot \nu &= 0 &\text{ a.e. in } \partial \Omega \times (0,\Tm),  \\
h_l(0) &= h_0 &\text{ in } H^1(\Omega). \label{haw}
\end{align}
\noindent
It holds either $\Tm = \infty$ or $\Tm < \infty$ and
\begin{align} \label{limt}
\lim\limits_{t \nearrow \Tm} \left(\|u_l(\cdot,t)\|_{L^{\infty}(\Omega)} + \|h_l\|_{W^{1,{\infty}}(\Omega)} \right)  = \infty.
\end{align}
\end{Lemma}

\begin{proof}
Fix $l \in \N$. Due to the Stone-Weierstrass theorem there is a sequence $(u_{0k})_{k \in \N} \subset C^{0,1}(\oa)$, $u_{0k} \geq 0$ with limit $u_0$ in $C(\oa)$. We set $M := \sup_{k \in \N} \|u_{0k}\|_{L^{\infty}(\Omega)} < \infty$. 
For $h<0$ and $\bu \geq 0$ extend the coefficients by
$$g(\bu,h) := 2 g(\bu,0) - g(\bu,-h) \text{ and } \mu(h) := \mu(-h).$$ 
We show the existence of a solution $(u_{lk},h_{lk})$ of \eqref{IBVPL} with initial value $u_{0k}$ instead of $u_0$ in the sense of \eqref{wsu} and \eqref{wsh} for  $k \in \N$ by showing the existence of a fixed point of the operator $F$ introduced below similarly to \cite{TW}.
Namely, we define for some small enough $T>0$ the set
$$ S := \{ \bu \in L^{\infty}(\ot )\ : \, 0 \leq \bu \leq M+1 \text{ a.e. in } \ot \}.$$ For $\bu \in S$ we consider the IBVPs
\begin{align} \label{IBVPU}
\begin{cases}
\partial_t u_{lk} = \nabla\nabla:(\D(x)u_{lk}) + \nabla \cdot (\D(x)u_{lk} \nabla h_{lk}) + \mu(h_{lk}) \bu^{\alpha-1} (1- J \ast \bu^\beta )u_{lk}, & x\in \Omega, \, t \in (0,T),\\
(\D(x)\nabla u_{lk} + \nabla \cdot \D(x)u_{lk}+\D (x) u_{lk} \nabla h_{lk}) \cdot \nu = 0, & x \in \partial \Omega, \, t \in (0,T),\\
u_{lk}(x,0) = u_{0k}(x), & x \in \Omega,
\end{cases}
\end{align}
and
\begin{align} \label{IBVPH}
\begin{cases}
\partial_th_{lk} = D_H \Delta h_{lk} + g(\bu,h_{lk}), & x \in \Omega, \, t \in (0,T),\\
\nabla h_{lk} \cdot \nu = 0, & x \in \partial \Omega, \, t\in (0,T),\\
h_{lk}(x,0) = h_0(x), & x \in \Omega.
\end{cases}
\end{align}
Here, $T$ can be chosen independent of $\bu$ and $k$. Through a fixed point argument similar to \cite{HW}, we conclude that there is a unique function $h_{lk} \in L^{\infty}(0,T;W^{1,\infty}(\Omega))$ that satisfies
\begin{align} \label{semigrouph} 
h_{lk} = e^{tD_H \Delta}h_0 + \int_0^t e^{(t-s)D_H \Delta} g(\bu,h_{lk}) \, \td s.
\end{align}
Moreover, $h_{lk}$ is the unique weak solution of \eqref{IBVPH} in the sense that
for a.e. $t \in (0,T)$ and all $\eta \in W^{1,1}_2(\ot)$ it holds that
\begin{align*}
&\int_{\Omega} h_{lk}(x,t) \eta(x,t) \,dx - \int_0^t \int_{\Omega}  h_{lk} \eta_t \,dx\,ds + D_H \int_0^t \int_{\Omega} \nabla h_{lk} \cdot \nabla \eta \,dx\,ds\\
= &\int_0^t \int_{\Omega} g(\bar{u},h_{lk}) \eta \,dx\,ds + \int_{\Omega} h_0(x) \eta(x,0) \, dx.
\end{align*}  
\noindent
As in Lemma \ref{boundgh} below and due to Theorem IV.9.1 (and the remark at the end of that section) in \cite{LSU}, it follows that
\begin{align} \label{ghw1q}
\|\nabla h_{lk} \|_{L^{\infty}( \Omega \times (0,T))} \leq \Cl{gh}.
\end{align}
and
\begin{align} \label{hw212}
\|h_{lk}\|_{W^{2,1}_2(\ot)} \leq \Cl{hw212}.
\end{align}

\noindent
Hence, $h_{lk}$ solves \eqref{IBVPH} in the sense of \eqref{wsh}. Moreover, the continuity of $h_{lk}$ follows from Theorem 4 in \cite{DB} due to $W^{2,1}_2(\ot) \subset C(0,T;L^2(\Omega))$ and the embedding of $W^{1,\infty}(\Omega)$ into some H\"older space on $\Omega$. 
Now, Theorems III.5.1 and 7.1 in \cite{LSU} (that also hold for our no-flux boundary condition), Theorem 4 in \cite{DB}, and Gronwall's inequality imply that there is a unique $u_{lk}$ in the space $S \cap \ck$, such that it solves \eqref{IBVPU} in the sense of \eqref{wsu} (with $u_{0k}$ instead of $u_0$) and satisfies
\begin{align} \label{ukck}
\|u_{lk}\|_{\ck} \leq \Cl{ck}(l)
\end{align}
for some $\kappa \in (0,1)$. Note that $\Cr{gh}$, $\Cr{hw212}$ and $\Cr{ck}$ are independent from $\bu$ and $k$. Hence, the operator 
$$F:S \mapsto S, \quad \bu \mapsto u_{lk},$$ 
where $u_{lk}$ solves \eqref{IBVPU} for $\bu$ in the sense of \eqref{wsu}, is well-defined. Moreover, as $\ck \hookrightarrow \hookrightarrow C(\oa \times [0,T])$ due to the Arzel\`a-Ascoli theorem, $F$ maps bounded sets on precompact ones. To apply the Leray-Schauder theorem it remains to show that $F$ is closed and, consequently, a compact operator. Let
\begin{align}
\bu_m &\underset{m \to \infty}{\rightarrow} \bu \text{ in } L^{\infty}(\ot) \label{convub}\\
u_{lkm}:= F(\bu_m) &\underset{m \to \infty}{\rightarrow} u_{lk} \text{ in } L^{\infty}(\ot). \label{convu}
\end{align}
We want to show that $F(\bu) = u_{lk}$. \\[-2ex]

\noindent
Let $h_{lkm}$ be the solution of \eqref{IBVPH} that corresponds to $\bu_m$ for $m \in \N$. Due to \eqref{hw212} we conclude from the Lions-Aubin lemma and the Banach-Alaoglu theorem that there are $h_{lk} \in W^{2,1}_2(\ot)$ and subsequences
\begin{align}
h_{lkm_{o}} &\underset{o \to \infty}{\rightharpoonup} h_{lk} \text{ in } L^2(0,T;H^2(\Omega)), \nonumber\\
h_{lkm_{o}} &\underset{o \to \infty}{\rightarrow} h_{lk} \text{ in } L^2(0,T;H^1(\Omega)) \text{ and a.e. in } \ot, \label{konvhl2}\\
\partial_t h_{lkm_{o}} &\underset{o \to \infty}{\rightharpoonup} \partial_t h_{lk} \nonumber \text{ in } L^2( \Omega \times (0,T)).
\end{align} 
Therefore, due to \eqref{konvhl2} and the Lipschitz-continuity of $g$ for a.e. $t \in (0,T)$ and all $\eta \in W_2^{1,1}(\ot)$ it holds that $h_{lk}$ is a solution of \eqref{IBVPH} in the sense of \eqref{wsh}.\\[-2ex]

\noindent
From \eqref{wsu} we conclude as in \cite{LSU} that for a.e. $t\in (0,T)$ and all $\eta \in W_2^{1,1}(\ot)$ it holds that 
\begin{align} \label{gronwallu}
&\frac{1}{2} \|u_{lkm}(\cdot,t)\|_{L^2(\Omega)}^2 +  \int_0^t \int_{\Omega} \left( \nabla \cdot \D_l u_{lkm} + \D_l \nabla u_{lkm} + \D_l u_{lkm} \nabla h_{lkm}\right) \cdot \nabla u_{lkm} \, \td x\, \td s \nonumber\\
&= \int_0^t\int_{\Omega} \mu(h_{lkm})\bu_m^{\alpha-1}(1-J\ast \bu_m^{\beta})u_{lkm}^2 \, \td x \, \td s + \frac{1}{2}\|u_{0k}\|_{L^2(\Omega)}^2.
\end{align}
Using H\"older's and Young's inequalities and \eqref{ghw1q}, we estimate 
\begin{align}
 \left|\int_{\Omega} \D_l u_{lkm} \nabla h_{lkm} \cdot \nabla u_{lkm} \, \td x \right|
&\leq \Cr{gh}\|\D_l\|_{L^\infty(\Omega)} \|u_{lkm}\|_{L^2(\Omega)} \|\nabla u_{lkm}\|_{L^2(\Omega)}\notag \\
& \leq \Cl{constduh}{(l)}\|u_{lkm}\|_{L^2(\Omega)}^2 + 
\frac{D_1}{2} \| \nabla u_{lkm} \|_{L^2(\Omega)}^2.
\label{abschnhgn}
\end{align}
Inserting this into \eqref{gronwallu} and using \eqref{ghw1q}, Young's inequality, the continuity of $h_{lk}$ and the Lipschitz-continuity of $\mu$, we conclude that for a.e. $t \in (0,T)$ it holds 
\begin{align*}
&\frac{1}{2} \|u_{lkm}(\cdot,t)\|_{L^2(\Omega)}^2 + \frac{D_1}{4} \int_0^t \| \nabla u_ {lkm} \|_{L^2(\Omega)}^2 \, \td s \leq \Cl{d1dijnaz}{(l)}  \int_0^t \|u_{lkm}\|_{L^2(\Omega)}^2 + \frac{1}{2}\|u_{0k}\|_{L^2(\Omega)}^2.
\end{align*}
From Gronwall's inequality we obtain a constant $\Cl{gwu}{(l,{\cb k,T})}>0$ 
such that $\|\nabla u_{lkm}\|_{L^2(0,T;\Omega)} \leq \Cr{gwu}{(l,{\cb k,T})}$ for all $m \in \N$.\footnote{The majority of subsequent constants will depend on $T$, but we will omit it in the writing.}
 Hence, the Banach-Alaoglu theorem implies that (by switching to a subsequence, if necessary)
\begin{align}
\nabla u_{lkm_{o}} \underset{{o} \to \infty}{\rightharpoonup} \nabla u_{lk} \text{ in } L^2(\ot). \label{convux}
\end{align}
Then, \eqref{convub} - \eqref{konvhl2}, \eqref{convux} and the dominated convergence theorem imply that $u_{lk}$ is a solution of \eqref{IBVPU} in the sense of \eqref{wsu}, and therefore $F(\bu ) = u_{lk}$ and $F$ is a compact operator. Consequently, by a Leray-Schauder argument we obtain the existence of a fixed point $u_{lk}$ of $F$, that satisfies for a.e. $t \in (0,T)$ and all $\eta \in W_2^{1,1}(\Omega)$ the weak formulation \eqref{wsu} for $u_{0}$ replaced by $u_{0k}$. \\[-2ex]

\noindent
Now, \eqref{ukck} and the compact embedding of $\ck$ in $C(\oa \times [0,T])$ imply that there is a convergent subsequence of $(u_{lk})_k$ such that
\begin{align*}
u_{lk_o}  \underset{o \to \infty}{\rightarrow} u_{l} \text{ in } C(\oa \times [0,T]).
\end{align*}
Then, with the same arguments as before we obtain the desired weak solution $(u_{l},h_{l})$ of \eqref{IBVPL}. \\[-2ex]

\noindent
Finally, for such pair property \eqref{limt} follows from a standard extensibility argument.
\end{proof}

%

\begin{Theorem} \label{exloc}
There is $\Tm \in (0, \infty]$ and a unique solution $(u_l,h_l)$ of \eqref{IBVPL} with $0 \leq u_{l}$ and $0 \leq h_{l} {<} H$,
\begin{align} \label{ssol}
u_{l},h_{l} \in C(\oa \times [0, \Tm)) \cap C^{2,1}( \oa \times (0, \Tm)).
\end{align}
\end{Theorem}

\begin{proof}
1. \textit{Regularity:} Let ${l \in \N},$ $0< T_1<\Tm $ and consider the weak solution $(u_l,h_l)$ from Lemma \ref{wsollem}. Again from Theorem 4 in \cite{DB} it follows that $u_l, h_l \in C^{\lambda, \frac{\lambda}{2}}(\oa \times (0,T_1])$ for some $\lambda \in (0,1)$. Combining this with the Lipshitz continuity of $g$, Theorem III.12.2 in \cite{LSU} implies $h_l \in C^{2+\lambda, 1+\frac{\lambda}{2}}(\Omega\times (0,T_1))$. \\[-2ex]

\noindent
Further, we know from Lemma \ref{wsollem} that $h_l \in W_2^{2,1}(\ot) \cap C(\oa \times [0,T_1])$ and hence, for a.e. $T_0 \in (0,T_1)$ it holds that $h_l(\cdot,T_0) \in H^2(\Omega) \cap L^{\infty}(\Omega)$ and $\nabla h_l(\cdot, T_0) \cdot \nu = 0$ a.e. on $\partial \Omega$. 

\noindent
If $ h_l(\cdot, T_0) \in W_r^2(\Omega) \cap  L^{\infty}(\Omega)$ for some $r \in (1,\infty)$, from the Gagliardo-Nirenberg inequality (see e.g. Corollary 5.1 in \cite{BM}) it follows that $h_l(\cdot,T_0) \in W_{r+1}^{2-\frac{1}{r+1}}(\Omega)$. Then, Theorem IV.9.1 in \cite{LSU} implies that $h_l \in W_{r+1}^{2,1}(\Omega \times (T_0,T_1))$. We start with $r=2$ and apply this procedure iteratively until $h_{l} \in W_{r^*}^{2,1}(\Omega \times (T_0,T_1))$ for some $r^* > n+2$. Then, $W_{r^*}^{2,1}(\Omega \times (T_0,T_1)) \hookrightarrow C^{1+\lambda^*,\frac{1+ \lambda^*}{2}}(\oa \times [T_0,T_1])$ for some $\lambda^* \in (0,1)$. Consequently, $h_{l} \in C^{2,1}(\Omega \times (0,T_1))$ with $\nabla h \in C(\oa \times (0,T_1])$ is a classical solution of the heat equation in \eqref{IBVPL}. Finally, Theorem 5.18 in \cite{L} implies that $h_{l} \in C^{2,1}(\oa \times (0, \Tm))$.
This follows analogously for $u_{l}$, thereby using Theorem 2.1 from \cite{DHP} instead of Theorem IV.9.1 in \cite{LSU}. \\[-2ex]

\noindent
The boundedness and nonnegativity of $h_l$ ($0 \leq h_l < H$) follow from the comparison principle of the semilinear heat equation with Neumann boundary condition and our assumptions on $g$.\\[-2ex] 

\noindent
2. \textit{Uniqueness:} With an ansatz similar to \cite{BCL} we want to show the uniqueness of the solution. Assume that there are two solutions $(u_{l,1},h_{l,1}), \, (u_{l,2}, h_{l,2})$ of \eqref{IBVPL} satisfying \eqref{ssol}. 
The functions $h_{l,1}$ and $h_{l,2}$ satisfy 
\begin{align*}
	\partial_t (h_{l,1} -h_{l,2}) = D_H \Delta (h_{l,1} - h_{l,2}) + g(u_{l,1},h_{l,1}) -g(u_{l,2},h_{l,2})
\end{align*}
in $\Omega\times (0,T_1)$. We multiply this equation with $h_{l,1} -h_{l,2}$ and integrate over $\Omega$. Then, using the boundary condition, the Lipshitz continuity of $g$, and Young's and Gronwall's inequalities, we conclude 
\begin{align} \label{normh1h2}
\|h_{l,1} - h_{l,2}\|_{L^{\infty}(0,T_1;L^2(\Omega))},\ \|\nabla(h_{l,1}-h_{l,2})\|_{L^2(\Omega \times (0,T_1))} \leq \Cl{ch12}(T_1) \|u_{l,1}-u_{l,2}\|_{L^2(\Omega \times (0,T_1)))}.
\end{align}

\noindent
Moreover, we can rewrite
\begin{align*}
\partial_t(u_{l,1}-u_{l,2}) =& \nabla \nabla :(\mathbb{D}_{l}(u_{l,1}-u_{l,2})) + \nabla \cdot(\mathbb{D}_{l}(u_{l,1}-u_{l,2})\nabla h_1) +  \nabla \cdot (\mathbb{D}_{l} u_{l,2}\nabla (h_{l,1}-h_{l,2})) \\ 
& + (\mu(h_{l,1})-\mu(h_{l,2}))u^{\alpha}_1(1-J\ast \ub_1) + \mu(h_{l,2}) (\ua_1-\ua_2)(1-J\ast \ub_1) + \mu(h_{l,2}) \ua_2J\ast (\ub_2-\ub_1).
\end{align*}
Again, we multiply this equation with $u_{l,1} - u_{l,2}$ and integrate over $\Omega$ for $t \in (0,T_1)$. Then, using the boundary condition together with Young's, H\"older's, and the Gagliardo-Nirenberg inequalities (also compare \eqref{abschnhgn}), the mean value theorem, and the boundedness of $u_{l,1}$ and $u_{l,2}$ on $\Omega \times (0,T_1)$ by some $\Cl{bu}(T_1, l)>0$, it follows that
\begin{align*}
&\frac{1}{2} \frac{d}{dt} \|u_{l,1} -u_{l,2}\|_{L^2(\Omega)}^2 + D_1 \|\nabla(u_{l,1} - u_{l,2})\|_{L^2(\Omega)}^2 \nonumber\\
\leq& - \int_{\Omega} \nabla \cdot (\mathbb{D}_{l}(u_{l,1}-u_{l,2}))\nabla (u_{l,1}-u_{l,2}) + (u_{l,1}-u_{l,2})(\mathbb{D}_{l}\nabla h_{l,1})\cdot \nabla (u_{l,1}-u_{l,2}) \notag \\
&+u_{l,2}(\mathbb{D}_{l}\nabla (h_{l,1}-h_{l,2}) ) \cdot \nabla (u_{l,1}-u_{l,2}) \, \td x\\
&+ \int_{\Omega}[(\mu(h_{l,1})-\mu(h_{l,2}))u^{\alpha}_1(1-J\ast \ub_1) + \mu(h_{l,2}) (\ua_1-\ua_2)(1-J\ast \ub_1)\notag \\
& + \mu(h_{l,2}) \ua_2J\ast (\ub_2-\ub_1)] (u_{l,1}-u_{l,2}) \td x\\
\leq& \Cl{ab1u1u2}(l)\left( \|u_{l,1}-u_{l,2}\|_{L^2(\Omega)}^2 + \|\nabla (h_{l,1}-h_{l,2})\|_{L^2(\Omega)}^2\right) + \frac{3 D_1}{4} \|\nabla(u_{l,1}-u_{l,2})\|_{L^2(\Omega)}^2\notag \\
&+ L_{\mu}\Cr{bu}(l)(1+\|J\|_{L^1}\Cr{bu}(l)^{\beta}) \|h_{l,1} -h_{l,2}\|_{L^1(\Omega)}\notag \\
&+ (L_{\mu} H +\mu(0))\left(\alpha \Cr{bu} (l)^{\alpha-1} (1+\|J\|_{L^1}\Cr{bu}(l)^{\beta}) +2 \Cr{bu}(l)^{\alpha+ \beta} \|J\|_{L^1}  \right) \|u_{l,1}-u_{l,2}\|_{L^1(\Omega)}.
\end{align*}
Integrating over $(0,t)$ for $t \in (0,T_1)$ and using \eqref{normh1h2} we conclude that for a.e. $t \in (0,T_1)$ it holds that
\begin{align*}
\|u_{l,1}-u_{l,2}\|_{L^2(\Omega)}^2 \leq& \Cl{ab2u1u2} (T_1,l) \left(\int_0^t \|u_{l,1}-u_{l,2}\|_{L^2(\Omega)}^2 + \|\nabla(h_{l,1}-h_{l,2})\|_{L^2(\Omega)}^2 + \|h_{l,1}-h_{l,2}\|_{L^2(\Omega)}^2 \right)\\
\leq& \Cl{ab3u1u2}(l) \int_{0}^t \|u_{l,1}-u_{l,2}\|_{L^2({\Omega})}^2.
\end{align*}
Consequently, Gronwall's inequality implies $u_1 \equiv u_2$ a.e. on $\Omega \times (0,T_1).$
\end{proof}

\subsection{Global existence and boundedness of $u$ in the approximate problem}
\begin{Lemma} \label{boundgh}
It holds that
\begin{align*}
\|\nabla h_{l}\|_{L^{\infty}(\Omega\times (0,\Tm))} \leq \Cl{cboundgh}.
\end{align*}
\end{Lemma}
\begin{proof}
With Lemma 1.3 from \cite{W2} and \eqref{semigrouph} we estimate for $t \in (0,\Tm)$ that
\begin{align*}
\|h_l(t)\|_{L^{q}(\Omega)} \leq \Cl{cw13}\left( e^{-\lambda_1D_Ht} \|h_0\|_{L^q(\Omega)} + \int_0^t \left(1+\frac{1}{(D_H(t-s))^{\frac{1}{2}}}\right)e^{-\lambda_1D_H(t-s)}\|g(u_l,h_l)\|_{L^{\infty}} \right) \td s
\end{align*}
holds for $q \in (1,\infty)$, where $\lambda_1$ is the first eigenvalue of $-\Delta$ on $\Omega$ with Neumann boundary condition. Using the properties of $g$ and the boundedness of $h$, we obtain that for $t \in (T_0 ,\Tm)$ it holds that
\begin{align*}
\|\nabla h_l\|_{L^{q}(\Omega)}
\leq& \Cr{cw13} e^{-\lambda_1 D_H t} \|\nabla h_0\|_{L^q(\Omega)} + \Cr{cw13}(\|\partial_hg\|_{L^{\infty}}H+G) \int_0^t \left(1+\frac{1}{(D_H(t-s))^{\frac{1}{2}}}\right)e^{-\lambda_1D_H(t-s)} \, \td s\\
\leq& \Cr{cw13}\|\Omega|^{\frac{1}{q}}\|\nabla h_0\|_{L^{\infty}(\Omega)}+ \frac{\Cr{cw13}(\|\partial_hg\|_{L^{\infty}}H+G) }{D_H}\left(\frac{1}{\lambda_1} + \frac{\sqrt{\pi}}{\sqrt{D_H}} \right) .
\end{align*}
Consequently,
\begin{align*}
\|\nabla h_l\|_{L^{\infty}(\Omega)} =& \lim\limits_{q\to\infty}\|\nabla h_l\|_{L^{q}(\Omega)} \\
\leq& \Cr{cw13}\|\nabla h_0\|_{L^{\infty}(\Omega)}+ \frac{\Cr{cw13}(\|\partial_hg\|_{L^{\infty}}H+G) }{D_H}\left(\frac{1}{\lambda_1} + \frac{\sqrt{\pi}}{\sqrt{\lambda_1}} \right) =: \Cr{cboundgh}
\end{align*}

\end{proof}

\noindent
We will show the global boundedness of $u$ as in the proof of Theorem 1.1 in \cite{LCS}.

\begin{Lemma} \label{ulq}
It holds that $u_l \in L^{\infty}(0,\Tm;L^q(\Omega))$ for $q \in [1, \infty)$. 
\end{Lemma}
\begin{proof}
Let $q \geq \max\{1, \beta + \alpha-1 \}$. Due to \eqref{ssol} the terms in the estimates below are well-defined for a.e. $t \in (0,\Tm)$. 
Multiplying the first equation of \eqref{IBVPL} by $q u_l^{q-1}$, integrating over $\Omega$ and using partial integration, we obtain
\begin{align}
\frac{d}{dt}\int_{\Omega} u_l^q \, dx =& q\int_{\Omega} \nabla \cdot \left(\D_l \nabla u_l + \nabla \cdot \D_l u_l^{q-1} + \D_l u_l \nabla h_l \right) u_l^{q-1} + \mu(h) \uaq_l (1- J \ast \ub_l) \,\td x \nonumber\\
=& - q(q-1) \int_{\Omega} u_l^{q-2} (\D_l \nabla u_l) \cdot \nabla u_l + u_l^{q-1} \nabla \cdot \D_l \cdot \nabla u_l +u_l^{q-1} (\D_l \nabla h_l) \cdot \nabla u_l \td x \nonumber \\
&+ q\int_{\Omega} \mu(h_l) \uaq_l (1- J \ast \ub_l) \, dx. \label{abschup1}
\end{align}
\noindent
Using the uniform parabolicity of $\mathbb{D}$, we estimate
\begin{align*}
q(q-1) \int_{\Omega} u_l^{q-2} (\D_l \nabla u_l) \cdot \nabla u_l \, \td x = \frac{4(q-1)}{q} \sum_{i,j = 1}^n \int_{\Omega} d_{_lij} \left(u_l^{\frac{q}{2}}\right)_{x_i}\left(u_l^{\frac{q}{2}}\right)_{x_j} \, \td x \geq \frac{4(q-1)}{q}{D_1} \int_{\Omega} |\nabla u_l^{\frac{q}{2}}|^2 \, \td x.
\end{align*}

\noindent
Further, due to Young's inequality we obtain the estimate
\begin{align*}
&q(q-1) \left|\int_{\Omega} u^{q-1} \nabla \cdot \D \cdot \nabla u + u^{q-1}(\D\nabla h) \cdot \nabla u \, \td x \right| \\
\leq &\frac{2(q-1)}{q }D_1 \int_{\Omega} |\nabla u^{\frac{q}{2}}|^2 \, \td x + \frac{(q-1)}{D_1}\left(\|\nabla \cdot \D\|_{\infty}^2 + \|\D\|_{\infty}^2 \|\nabla h\|_{\infty}^2  \right) \int_{\Omega} u^q \, \td x.
\end{align*}

\noindent
Inserting these estimates into \eqref{abschup1} and using our assumptions on $\mu$ and $J$, the boundedness of $h$, and 
\begin{align} \label{abschuqa}
\int_{\Omega} u_l^q \, \td x \leq \int_{\Omega} u_l^{q + \alpha -1} \, \td x + |\Omega|,
\end{align}
it follows that
\begin{align} \label{abschup2}
\frac{d}{dt} \|u_l\|_{L^q}^q + \frac{2(q-1)}{q }D_1 \int_{\Omega} |\nabla u_l^{\frac{q}{2}}|^2 \, \td x + q\eta \delta \int_{\Omega} \uaq_l \, \td x \int_{\Omega} u_l^{\beta} \, \td x \leq  q \Cl{cwk}(q) \left( \int_{\Omega} \uaq_l \, dx +|\Omega|\right),
\end{align}
where 
\begin{align*}
\Cr{cwk}(q,l) :=  \frac{q-1}{D_1} \left(\|\nabla \cdot \D_l\|_{\infty}^2 + \|\D_l\|_{\infty}^2 \Cr{cboundgh}^2\right) + L_{\mu} H + \mu(0).
\end{align*}
\noindent
Adding $q\Cr{cwk}(q,l) \|u_l\|_{L^q}^q$ on both sides of \eqref{abschup2} and using Young's inequality one more time, we obtain
\begin{align}
&\frac{d}{dt}\int_{\Omega} u_l^q \, \td x + q\Cr{cwk}(q,l) \|u_l\|_{L^q}^q + 2\frac{q-1}{q}D_1\int_{\Omega} |\nabla u_l^{\frac{q}{2}}|^2 \, \td x + q\delta \eta \int_{\Omega} \uaq_l \, \td x \int_{\Omega} \ub_l \, \td x \nonumber\\
\leq& 2q\Cr{cwk}(q,l) \left( \int_{\Omega} \uaq_l +|\Omega|\right).\label{abschup3}
\end{align}
\noindent
Similarly to Step 1 in the proof of Theorem 1.1 in \cite{LCS} it follows that
\begin{align*}
2q\Cr{cwk}(q, l) \io u_l^{q-1+\alpha} \leq 2\frac{q-1}{q} D_1 \io  |\nabla u_l^{\frac{q}{2}}|^2 \, dx + q\delta \eta \int_{\Omega}  \uaq_l  \, dx \int_{\Omega}  \ub_l  \, dx + 2q\Cr{cwk}(q,l) \Cl{cwk2}(q, l),
\end{align*}
where 
\begin{align*}
\Cr{cwk2}(q,l):=& \Cl{c2}(q)^{\frac{s(q+\alpha-1+\beta)}{qs-(q+\alpha-1+\beta)}}+\left(2\left(\frac{2\Cl{c1}^2q^2 \Cr{cwk}(q,l)}{(q-1)D_1} \right)^{\frac{q+\alpha-1+\beta}{(q+\alpha-1+\beta)(1-\frac{1}{s})-2(\alpha-1)}} \right.\\
&\left.+ \Cr{c2}(q) ^{\frac{s(q+\alpha-1+\beta)}{qs-(q+\alpha-1+\beta)}} \right)^{\frac{(s-2)(q+\beta)-(s+2)(\alpha-1)}{s(\beta+1-\alpha)-2\beta}}\left(\frac{2\Cr{cwk}(q,l)}{\delta \eta}\right)^{\frac{sq-2(q+\alpha-1)}{s(\beta+1-\alpha)-2\beta}}
\end{align*}
with
\begin{align*}
\Cr{c2} (q) &:= 4C_S|\Omega|^{\frac{1}{2}-\frac{q}{q+\alpha-1+\beta}}\\
\Cr{c1} &:= 2C_S^2(1+2C_P).
\end{align*}
Here, $C_S >0$ denotes the Sobolev embedding constant from page 8 in \cite{LCS}, $C_P>0$ the constant from the Poincar\'e inequality, and 
\begin{align} \label{defs}
s \begin{cases}
=\infty, & n=1\\
\in (\frac{2(q+\alpha-1+\beta)}{q-\alpha+1+\beta},\infty), &n=2\\
=\frac{2n}{n-2}, & n>2.
\end{cases}
\end{align}
\noindent
Hence, for $t \in (0,\Tm)$ we conclude that 
\begin{align} \label{abschup4}
\frac{d}{dt}\|u_l\|_{L^q(\Omega)}^q + q \Cr{cwk}(q,l)\|u\|_{L^q(\Omega)}^q \leq 2q\Cr{cwk}(q,l)(\Cr{cwk2}(q,l) +|\Omega|).
\end{align}
\noindent
Hence, for $t \in (0,\Tm)$ and $q \geq \max\{\beta +\alpha-1,1\}$ we obtain from \eqref{abschup4} the upper bound
\begin{align} \label{ublq}
\|u_l(\cdot,t)\|_{L^q(\Omega)} \leq \sqrt[q]{2\Cr{cwk2}(q, l) + |\Omega| (2 +\|u_0\|_{L^{\infty}(\Omega)}^q)}.
\end{align}
\end{proof}

\begin{Remark}
As in \cite{LCS} we cannot directly conclude from Lemma \ref{ulq} that $u_{l}$ is bounded on $\Omega \times (0,\Tm)$ as 
\begin{align*}
\lim\limits_{q \to \infty} \sqrt[q]{2\Cr{cwk2}(q, l) + |\Omega| (2 +\|u_0\|_{L^{\infty}(\Omega)}^q)}= \infty.
\end{align*}
\end{Remark}

\begin{Theorem} \label{globex}
For all $l\in\N$ there is a unique bounded and nonnegative solution $(u_l,h_l)$ of \eqref{IBVPL} consisting of nonnegative functions 
\begin{align*} 
u_l,h_l \in C(\oa \times [0, \infty)) \cap C^{2,1}( \oa \times (0, \infty))
\end{align*}
and $h_l < H$. Thereby, there is some $\Cl{ubu} >0$ that does not depend on $l$ s.t. $u_l \leq \Cr{ubu}$.
 
\noindent
Moreover, if $\Omega$ is convex for some 'small' enough choice of parameters specified in \eqref{bed1} and \eqref{bed2} below, then for any $K>1$ 
\begin{align}\label{boundu}
\|u_l\|_{L^{\infty}(\Omega\times(0,\infty))}\leq K \max\left\{1,\|u_0\|_{L^{\infty}(\Omega)},\frac{2}{\delta \eta}\Cl{lc2}^{\frac{s-2}{s(\beta+1-\alpha)-2\beta}}\right\},
\end{align}
where
\begin{align*}
\Cr{lc2} := 2\sqrt{2} \max\{|\Omega|^{\frac{1}{s}-1},|\Omega|^{-\frac{3}{2}}\text{diam}(\Omega)^{1+\frac{2n}{s}}G(s,n)\}
\end{align*}
for some function $G$ to be adequately chosen and we set $s$ from \eqref{defs} equal to $\infty$ for $n=2$.
\end{Theorem}
\begin{proof}
Let $l\in\N$. We proceed with a Moser iteration as in Step 2 in Theorem 1.1. in \cite{LCS}. 

\noindent
Set $q_k := 2^k + a$ with $a:= \frac{2(s-1)(\alpha-1)}{s-2}$ for $k \in \N$ and $s$ as in \eqref{defs}. 
Analogously to \cite{LCS} we obtain for $t \in (0,\Tm)$ the estimate
\begin{align} \label{abschuqk}
\frac{d}{dt} \|u_l\|_{L^{q_k}}^{q_k} +q_k \Cr{cwk}(q_k, l) \|u_l\|_{L^{q_k}}^{q_k} \leq 2 q_k\Cr{cwk}(q_k, l)  \Cl{cwk3}(q_k) \max\left\{1,\|u_l\|_{L^{q_{k-1}}}^{2q_{k-1}} \right\},
\end{align}
where
\begin{align*}
\Cr{cwk3}(q_k) :&= |\Omega| + 2\left(\frac{2\Cr{c1}^2q_k^2 \Cr{cwk}(q_k)}{(q_k-1)D_1}\right)^{\frac{s}{s-2}} + 2\max\{\Cr{c2}(q_k),1\}^{\alpha+1}.
\end{align*}
As $\D_l$ tends to $\D$ in $C^1(\bar{\Omega})^{n\times n}$, there is $\Cl{bd}>0$ s.t. $\|\D_l\|_{C^1(\bar{\Omega})^{n\times n}} \leq \Cr{bd}$ for all $l \in \N)$.
We can further estimate that 
\begin{align*}
\Cr{cwk3}(q_k) \leq 2^{2k\frac{s}{s-2}}\Cl{a0}
\end{align*}
for
\begin{align*}
\Cr{a0} :=& |\Omega| + \max\{\Cr{c2}{(q_k)},1\}^{\alpha+1}\\ 
&+ 2\left(2\frac{\Cr{c1}^2}{D_1}(1+a)\left( \frac{1+a}{D_1} \left(n^2 \Cr{bd}^2 + \Cr{bd}^2\Cr{cboundgh}^2 \right) +L_{\mu}H+\mu(0) \right)\right)^{\frac{s}{s-2}}.
\end{align*}
For $k\in\N$ and $t\in(0,\Tm)$ we set
\begin{align*}
y_k(t) := \|u_l(\cdot,t)\|_{L^{q_k}}^{q_k}.
\end{align*}
Inserting this into \eqref{abschuqk} we obtain
\begin{align*}
y_k'(t) + q_k \Cr{cwk}(q_k, l)y_k(t) \leq 2q_k\Cr{cwk}(q_k, l) 2^{2k\frac{s}{s-2}}\Cr{a0}\max\left\{1, \left(\io u_l^{q_{k-1}} \right)^2 \right\}.
\end{align*}
\noindent
Moreover, we estimate 
that
\begin{align*}
\|u_0\|_{L^{q_k}(\Omega)}^{q_k} \leq \Cl{B}^{2^k}
\end{align*}
with
\begin{align} \label{defB}
\Cr{B}:=\max\{1,|\Omega|\}\max\{\|u_0\|_{L^{\infty}}^{\alpha},1 \}.
\end{align}

\noindent
Hence, from Lemma 2.1 in \cite{LCS} it follows that for $m \geq 1$ and $t \in (0,\Tm)$ it holds that
\begin{align*}
\io u_l^{q_k} \, dx \leq (4\Cr{a0})^{2^{k-m+1}}2^{\frac{2s}{s-2}\left(2(2^{k-m}-1)+m2^{k-m+1} -k\right)} \max\left\{\sup_{t \geq 0} \left(\io u_l^{q_{m}} \right)^{2^{k-m+1}} , \Cr{B}^{2^k},1 \right\}.
\end{align*}
\noindent
Consequently, for $t \in (0, \Tm)$ and $m \geq 1$ it holds that
\begin{align} \label{ubinfty}
\|u_l\|_{L^{\infty}(\Omega)} =& \lim\limits_{k \to \infty}\|u_l\|_{L^{q_k}(\Omega)}
\leq (4\Cr{a0})^{2^{-m+1}}2^{\frac{2s(1+m)}{(s-1)^{m-1}}} \max\left\{\sup_{t \geq 0} \left(\io u_l^{q_{m-1}} \right)^{2^{-m+1}} , \Cr{B},1 \right\} =: \Cl{cubuinf}(m, l).
\end{align}
Due to \eqref{ublq} and 
\begin{align*}
\Cr{cwk}(q,l) \leq \frac{q-1}{D_1} (n^2\Cr{bd}^2 + \Cr{bd}^2 \Cr{cboundgh}^2) + L_{\mu} H + \mu(0),
\end{align*}
there is $\Cr{ubu}(m)>0$ s.t. $\|u_l\|_{L^{\infty}(\Omega)} \leq \Cr{ubu}(m)$ for all $l \in \N$.\\[-2ex]

\noindent
Consequently, $u_l$ is bounded on $\oa \times [0,\Tm)$. Combining this with the boundedness of $h_l$, Lemma \ref{boundgh} and \eqref{limt} in Theorem \ref{exloc}, $\Tm = \infty$ follows.\\[-2ex]

\noindent
If $\Omega$ is convex we proceed as in Step 3 of \cite{LCS}.  First, we fix some $m$ and choose our parameters sufficiently 'small' such that it holds that
\begin{align}\label{bed1}
\frac{2\Cr{c1}}{D_1}  \left( \frac{q_{m-1}}{D_1}\left(\|\nabla \cdot \D_l\|_{\infty}^2 + \|\D_l\|_{\infty}^2 \Cr{cboundgh}^2\right)+L_{\mu}H+\mu(0)\right)&\leq  \frac{2\Cr{c1}}{D_1} \left( \frac{q_{m-1}}{D_1}\left( n^2\Cr{bd}^2 + \Cr{bd}^2 \Cr{cboundgh}^2\right)+L_{\mu}H+\mu(0)\right)\notag \\
& < \frac{1}{q_{m-1}^2}      
\end{align}
and
\begin{align} \label{bed2}
\Cr{cwk}(q_{m-1}) \leq { \frac{q_{m-1}-1}{D_1} \left(n^2\Cr{bd}^2  + \Cr{bd}^2  \Cr{cboundgh}^2\right) + L_{\mu} H + \mu(0)} \leq 1.
\end{align} 
This depends on our choice of $\mathbb{D}$, $\mu$, $g$, $\|\nabla h_0\|_{L^{ \infty}}$, and $D_H$. 
Consequently, we conclude as in \cite{LCS} that for any $K>1$ we find 'small' enough parameters (satisfying \eqref{bed1} and \eqref{bed2} for some large $m$) such that \eqref{boundu} holds,
where due to Theorems 2.1 and 3.1 from \cite{MTSO} we have 
\begin{align*}
\Cr{lc2} := \lim\limits_{m \to \infty} \Cr{c2}(q_{m-1}) = 2\sqrt{2} \max\{|\Omega|^{\frac{1}{s}-1},|\Omega|^{-\frac{3}{2}}\text{diam}(\Omega)^{1+\frac{2n}{s}}G(s,n)\}
\end{align*}
for 
$$G(s,n) := \frac{\pi^{\frac{s}{n}}\Gamma(\frac{s-2}{2s}n)}{n\Gamma(\frac{s-1}{s}n)}\left(\frac{\Gamma(n)}{\Gamma(\frac{n}{2})}\right)^{\frac{s-2}{s}}$$ 
and we set $s$ from \eqref{defs} equal to $\infty$ for $n=2$.
\end{proof}

\subsection{Global existence and boundedness in the original problem}
\begin{Theorem} \label{globexbound}
There is a unique bounded and nonnegative weak solution $(u,h)$ of \eqref{IBVP} s.t. for a.e. $T >0$ it holds that $u \in C(0,T;L^2(\Omega)) \cap L^2(0,T;H^1(\Omega))$ with $\partial_t u \in L^2(0,T;(H^1(\Omega))^*)$ and $h \in C(0,T; H^1(\Omega)) \cap W^{2,1}_{2}(\ot)$ and $u$ satisfies \eqref{wsu} and \eqref{uaw} (with $T$ instead of $t$) for all $\eta \in W^{1,1}_2(\ot)$ and $h$ satisfies \eqref{wsh} - \eqref{haw} a.e. in $\Omega \times (0,\infty)$. \\[-2ex]

\noindent
Moreover, it holds that $h \leq H$ and $\nabla h \in L^{\infty}(\Omega \times (0,\infty))$.
If $\Omega$ is convex then $u$ satisfies \eqref{boundu} for the parameter choice from Theorem \ref{globex}.
\end{Theorem}

\begin{proof}
Let  $\varphi \in H^1(\Omega)$.
Obviously, for a.e. $T>0$ and each $l \in \N$ the function $u_l$ satisfies
\begin{align}
\int_{\Omega} \partial_tu_{l} \varphi \, \td x &= - \int_{\Omega} ( \D_l \nabla u_l + \nabla \cdot \D_l u_l + \D_l u_l\nabla h_l) \cdot \nabla \varphi \, \td x + \int_{\Omega} \mu (h_l) u_l^{\alpha} (1-J\ast u_l^\beta) \, \td x, \label{ulw}
\end{align}
as it is a classical solution. Due to Theorem IV.9.1 (and the remark at the end of that chapter) in \cite{LSU}, $h_l$ satisfies 
\begin{align} \label{chw212}
\|h_l\|_{W^{2,1}_2(\Omega\times (0,T))} \leq \Cl{hw212-b},
\end{align}
where $\Cr{hw212-b}>0$ is independent from $l$ due to the properties of $g$ and $h_l {\color{olive} <} H$ for all $l\in\N$.\\[-2ex]

\noindent
Setting $\varphi=u_l$ in \eqref{ulw} and using H\"older's and Young's inequalities, the facts that $\|\D_l\|_{C^1(\bar{\Omega})^{n\times n}} \leq \Cr{bd}$, and the uniform boundedness of $u_l$ from Lemma \ref{globex}, we can estimate that
\begin{align*}
\frac{1}{2}\frac{d}{dt}\|u_l\|_{L^2(\Omega)}^2 + D_1 \|\nabla u_l\|^2_{L^2(\Omega)} \leq  \Cr{bd} \|u_l\|_{L^2(\Omega)}(n+\Cl{boundgh)})\| \nabla u_l\|_{L^2(\Omega)} + \Cl{lt}.
\end{align*}
Consequently, from Gronwall's inequality follows
\begin{align*}
\|\nabla u_l\|_{L^2(\Omega\times (0,T))} \leq \Cl{boundnu}.
\end{align*}
Similarly it follows that
\begin{align*}
\|\partial_t u_l\|_{L^2(0,T;(H^1(\Omega))^*)} \leq \Cl{boundut}.
\end{align*}
Putting this together with the uniform boundedness of $u_l$ and \eqref{chw212}, the Lions-Aubin and Banach-Alaoglu theorems imply that there are $u \in C(0,T;L^2(\Omega)) \cap L^2(0,T;H^1(\Omega))$ with $\partial_t u \in L^2(0,T;(H^1(\Omega))^*)$ and $h \in C(0,T; H^1(\Omega)) \cap W^{2,1}_{2}(\ot)$ s.t. (after switching to a subsequence if necessary)
\begin{align}
u_l &\underset{l \rightarrow \infty}{\rightarrow} u &\text{ in } L^2(\Omega \times (0,T)) \text{ and pointwise a.e.}, \label{ulpw}\\
u_l &\underset{l \rightarrow \infty}{\rightharpoonup} u &\text{ in } L^2(0,T;H^1(\Omega)), \nonumber\\
(u_l)_t &\underset{l \rightarrow \infty}{\rightharpoonup} u_t &\text{ in } L^2(0,T;(H^1(\Omega))^*),\nonumber\\
h_l &\underset{l \rightarrow \infty}{\rightarrow} h &\text{ in } L^2(0,T;H^1(\Omega)) \text{ and pointwise a.e.},\nonumber\\
h_l &\underset{l \rightarrow \infty}{\rightharpoonup} h &\text{ in } L^2(0,T;H^2(\Omega)),\nonumber\\
(h_l)_t &\underset{l \rightarrow \infty}{\rightharpoonup} h_t &\text{ in } L^2(\Omega\times (0,T)).\nonumber
\end{align}

\noindent
From this, the dominated convergence theorem, the uniform boundedness of $u_l$, $h_l $ and $\nabla h_l$ from Lemma \ref{boundgh} and the Lipshitz-continuity of $\mu$ and $g$, it follows that $(u,h)$ solves \eqref{IBVP} in the required sense.\\[-2ex]

\noindent
The a.e. boundedness and nonnegativity of $u$ and $h$ follow from the pointwise convergence and the uniform boundedness and nonnegativity of $u_l$ and $h_l$. 
Uniqueness follows similarly to Theorem \ref{exloc}.
\end{proof}

 \section{Long time behavior}\label{sec:asymptotic}
We consider the long time behavior of our solution under the additional assumptions that we make from now on:
\begin{itemize}
	\item the domain $\Omega$ is convex,
	\item the parameters satisfy \eqref{bed1} and \eqref{bed2},
	\item $u_0 \not\equiv 0$,
	\item we extend $J$ by $0$ to $\R^n\setminus B_{\text{diam}(\Omega)}(0)$ and assume $\|J\|_{L^1(\R^n)} = 1$,
	\item there are $h^* \in [0,H]$ and constants $C_H>0$ and $C_U \geq 0$ s.t.
	\begin{equation}\label{assg}
		g(u,h)(h-h^*) \leq - C_H (h-h^*)^2 + C_U u^{\alpha-1} (u^{\beta}-1)^2
	\end{equation}
for $0\leq h\leq H$ and $0 < u \leq U$, where $U$ is some upper bound on $u$ (that exists and is independent from $l$, due to Theorem \ref{globex}),
	\item the parameters
	\begin{align*}
		C_B &:=\frac{1}{4D_1}\left(L_{\mu}H+\mu(0)\right)\left(\text{diam} (\Omega)\beta U^{\beta}\right)^2,\\
		C_A &:= \frac{C_U \Cr{bd}((\beta-1)U^{\beta} +1)}{4\delta \eta |\Omega|D_HD_1} -1- C_B,
	\end{align*}
where $\|\mathbb{D}\|_{L^{\infty}(\Omega)} \leq \Cr{bd}$ for all $l \in \N$ (see proof of Theorem \ref{globex}), satisfy $C_A^2>4C_B$, $C_A<0$, $C_B \in (0,1)$ and
	\begin{align*}
		C_B  < -\frac{C_A}{2} + \sqrt{\frac{C_A^2}{4}-C_B} .
	\end{align*}
\end{itemize}
Moreover, let
\begin{align*}
	M := \left\{\mathbb{D} \in \left(C^{2+\vartheta}(\bar{\Omega})\right)^{n\times n}: \, (\nabla \cdot \mathbb{D} ) \cdot \nu = 0 \text{ on } \partial \Omega, \, \nabla \cdot (\nabla \cdot \mathbb{D}) = 0 \text{ on } \Omega \right\}.
\end{align*}
\noindent
We assume that $\mathbb{D}$ is in the closure of $M$ in the $(C^1(\bar{\Omega}))^{n \times n}$-norm and that the sequence $(\mathbb{D}_l)_{l\in\N}$ from above is the sequence in $M$ that approaches $\mathbb{D}$.

\begin{Remark}
	The inequality \eqref{assg} implies that such $h^*$ is unique and $g(1,h^*) = 0$ holds.
\end{Remark}

\noindent
We proceed by combining the methods from \cite{LCS} and \cite{Kumar2021}.

\begin{Lemma} \label{limtl2}
	For all $l \in\N$ it holds that
	\begin{align}\label{convtinf}
		\lim_{t\to\infty} \| u_l( \cdot, t)^{\frac{\alpha-1}{2}}(u_l^{\beta}(\cdot, t) - 1) \|_{L^2(\Omega)} = \lim_{t\to\infty} \| h_l( \cdot, t) -h^* \|_{L^2(\Omega)} = 0.
	\end{align}
\end{Lemma}

\begin{proof}
Let $l \in \N$. We conclude from the strong maximum principle and the assumption $u_0 \not\equiv 0$ that $u_l>0$ holds in $\Omega\times (0,\infty)$. \\[-2ex]

\noindent	
As in \cite{LCS} we define $a(s) := \frac{s}{\beta} -\frac{1}{\beta} \ln(s) - \frac{1}{\beta}$ with $a(s) \geq 0$ for $s \in (0,\infty)$. By multiplying the equation for $u_l$ in \eqref{IBVPL} by $u_l^{\beta - 1} - u_l^{-1}$, integrating over $\Omega$ and using partial integration and our additional assumptions on $\D_l$, we obtain
\begin{align*}
	&\frac{d}{dt} \io a(u_l^{\beta}) \, \td x\\
	=& \io \nabla \cdot \left(\mathbb{D}_l \nabla u_l + \nabla \cdot \mathbb{D}_lu_l +\mathbb{D}_l u_l \nabla h_l \right)\left(u_l^{\beta -1} - u_l^{-1}\right) \, \td x
	+ \io \mu(h_l)\ua_l (1- J \ast \ub_l)(u_l^{\beta-1} - u_l^{-1}) \, \td x\\
	=& - \io \left(\D_l \nabla u_l + \nabla \cdot \mathbb{D}_l u_l +\D_l u_l \nabla h_l \right) \left((\beta -1)\ub_l + 1\right) \cdot \frac{\nabla u_l}{u_l^2} \, \td x
	+ \io \mu(h_l)\ua_l (1- J \ast \ub_l)(u_l^{\beta-1} - u_l^{-1}) \, \td x,
\end{align*}
where using again partial integration
\begin{align*}
	&\io  \nabla \cdot \mathbb{D}_l u_l ((\beta-1)\ub_l + 1)\cdot\frac{\nabla u_l}{u_l^2} \, \td x \\
	=& \io \nabla \cdot \mathbb{D}_l  ((\beta-1)u_l^{\beta-1} + u_l^{-1}) \cdot \nabla u_l \, \td x\\
	=& \io \nabla \cdot \mathbb{D}_l \cdot \nabla\left(\frac{\beta - 1}{\beta}\ub_l +\ln(u_l) \right) \, \td x\\
	=& - \io \nabla \cdot (\nabla \cdot \mathbb{D}_l) \left(\frac{\beta-1}{\beta} \ub_l + \ln(u_l)\right) \, \td x + \int_{\partial \Omega} \left(\frac{\beta-1}{\beta}\ub_l +\ln(u_l)\right) (\nabla \cdot\mathbb{D}_l) \cdot \nu\td \sigma = 0.
\end{align*}

\noindent
Hence, we can estimate using the positivity of $u_l$ that
\begin{align} \label{estau2}
	&\frac{d}{dt} \io a(\ub_l) \, \td x + \io \left(\frac{\nabla u_l}{u_l}\right)^T \D_l \frac{\nabla u_l}{u_l} \,\td x \nonumber \\
	\leq&\io \left((\beta-1)\ub_l + 1\right) \left|\left(\nabla h_l\right)^T \D_l \frac{\nabla u_l}{u_l} \right| \, \td x \notag \\
	&+ \io \mu(h_l) u^{\alpha-1}_l \left(1-J\ast \ub_l\right)\left(\ub_l-1\right) \, \td x.
\end{align}
\noindent
Then, we extend $u_l$ by 0 to $\R^n \setminus \Omega$ and proceed similarly to the proof of Proposition 3.1 in \cite{LCS} to obtain using H\"older's inequality that
\begin{align} \label{estau1}
	&\io \mu(h_l(x))u^{\alpha-1}_l(x)(1-J\ast \ub_l(x)) (\ub_l(x) -1) \, \td x \nonumber \\
\leq&\io \mu(h_l(x)) u_l^{\alpha-1}(x)  \left(\int_{\R^n} J(y) \, \td y - \io J(x-y) \ub_l(y) \, \td y  \right) \left(\ub_l(x) -1 \right) \, \td x \nonumber \\
\leq&\io \mu(h_l(x)) u_l^{\alpha-1}(x)  \left(\io J(x-y) (1- \ub_l(y)) \, \td y \right)\left(\ub_l(x) -1\right) \, \td x \nonumber \\
\leq&-\io \mu(h_l(x)) u_l^{\alpha-1}(x)  \io J(x-y) (\ub_l(x) - 1)^2 \, \td y \, \td x \nonumber \\ 
	&+ \io \mu(h_l(x)) u_l^{\alpha-1}(x) \io J(x-y) (\ub_l(x) - \ub_l(y))(\ub_l(x)-1) \, \td y \, \td x \nonumber \\
\leq&-(1-\varepsilon)\io \io \mu(h_l(x)) u_l^{\alpha-1}(x)  J(x-y) (\ub_l(x) - 1)^2 \, \td y \, \td x \nonumber \\
	&+ \frac{1}{4\varepsilon} \left( L_{\mu} H + \mu(0)\right) \io \io u_l^{\alpha-1}(x) J(x-y) (\ub_l(x) - \ub_l(y))^2 \, \td y \, \td x
\end{align}
for 
\begin{align*}
\varepsilon \in \left\{\max \left\{ - \frac{C_A}{2} - \sqrt{\frac{C_A^2}{4} - C_B} ;C_B \right\},\min \left\{ - \frac{C_A}{2} + \sqrt{\frac{C_A^2}{4} - C_B} ;1 \right\}\right\},
\end{align*}
where the interval on the right hand side is nonempty due to our assumptions on $C_A$ and $C_B$.\\[-2ex]

\noindent
Moreover, due to the convexity of $\Omega$ and using the uniform boundedness of $(u_l)_l$ by $U$ we can estimate 
\begin{align*}
\io \io u_l^{\alpha-1}(x) J(x-y) (\ub_l(x) - \ub_l(y))^2 \, \td y \, \td x 
\leq \left(|\text{diam}(\Omega)|\beta\right)^2 {U^{2\beta+\alpha-1}} \io \frac{|\nabla u_l|^2}{u_l^2} \, \td x.	
\end{align*}
\noindent
Now, inserting this in \eqref{estau1}, using our assumptions on $J$ and $\mu$, and the uniform boundedness of $(u_l)_l$, we conclude
\begin{align*}
	&\io \mu(h_l(x))u^{\alpha-1}_l(x)(1-J\ast \ub_l(x)) (\ub_l(x) -1) \, \td x  \\
	\leq& \frac{1}{\varepsilon} D_1 C_B \io \frac{|\nabla u_l|^2}{u_l^2} \, \td x - (1-\varepsilon)\delta \eta |\Omega| \io u^{\alpha-1}_l(\ub_l - 1)^2 \, \td x
\end{align*}
Inserting this into \eqref{estau2}, it follows that
\begin{align}\label{estau3}
	&\frac{d}{dt} \io a(u_l^{\beta}) \, \td x + D_1 \frac{\varepsilon - C_B}{\varepsilon}  \io \frac{|\nabla u_l|^2}{u_l^2} \, \td x + (1-\varepsilon)\delta \eta |\Omega| \io u^{\alpha-1}_l(\ub_l - 1)^2 \, \td x \nonumber\\
\leq&\io \left((\beta-1)\ub_l + 1\right) \left|\left(\nabla h_l\right)^T \D_l \frac{\nabla u_l}{u_l} \right| \, \td x \notag \\ 
\leq&D_1\frac{\varepsilon - C_B}{\varepsilon} \io \frac{|\nabla u_l|^2}{u_l^2} \, \td x +\frac{\varepsilon\left(\Cr{bd}((\beta - 1) U^{\beta} + 1) \right)^2}{4D_1\left(\varepsilon -  C_B \right)} \io |\nabla h_l|^2\, \td x.
\end{align}
\noindent
Multiplying the equation for $h_l$ by $h_l - h^*$ and using \eqref{assg}, we obtain
\begin{align*}
	\frac{1}{2} \frac{d}{dt} \io (h_l - h^*)^2 \, \td x + D_H \io |\nabla h_l|^2 \, \td x \leq -C_H \io (h-h^*)^2 \, \td x + C_U \io u^{\alpha-1}_l(\ub_l - 1)^2 \, \td x.
\end{align*}
Further, we multiply this by $\Cl{ceqh}:= \frac{\varepsilon\left(\Cr{bd}((\beta - 1) U^{\beta} + 1) \right)^2}{4D_1D_H\left(\varepsilon -  C_B \right)}$ and add it to \eqref{estau3} to obtain
\begin{align}
	&\frac{d}{dt} \left(\io a(\ub_l) \, \td x + \frac{1}{2} \Cr{ceqh} \io (h_l-h^*)^2 \, \td x \right)\\
	+& C_H \Cr{ceqh} \io (h_l - h^*)^2 \, \td x
	+ \left((1- \varepsilon) \delta \eta |\Omega| - C_U \Cr{ceqh} \right) \io u^{\alpha-1}_l(\ub_l - 1)^2 \, \td x \leq 0.
\end{align}
Due to our assumptions on $C_A, C_B$ and our choice of $\varepsilon$ it holds that $(1- \varepsilon) \delta \eta |\Omega| - C_U \Cr{ceqh} > 0$. Now, Gronwall's inequality implies \eqref{convtinf}.
\end{proof}

\noindent
Now we can conclude uniform convergence:
\begin{Theorem} \label{limtlinf}
	For all $l \in\N$ it holds that
	\begin{align}\label{alimtlinf}
		\lim_{t\to\infty} \| u_l( \cdot, t) - c \|_{L^{\infty}(\Omega)} = \lim_{t\to\infty} \| h_l( \cdot, t) -h^* \|_{L^{\infty}(\Omega)} = 0,
	\end{align}
	where $c \in \{0,1\}$ for $\alpha >1$ and $c=1$ for $\alpha=1$.
\end{Theorem}

\begin{proof}
As in Lemma 3.10 in \cite{TW2} we can conclude from Lemma \ref{limtl2} that
\begin{align*}
	\lim_{t\to\infty} \| u_l(t, \cdot)^{\frac{\alpha-1}{2}}(u_l^{\beta}(t,\cdot) - 1) \|_{L^{\infty}(\Omega)} = \lim_{t\to\infty} \| h_l(t, \cdot) -h^* \|_{L^{\infty}(\Omega)} = 0.
\end{align*}
Combining this with the uniform continuity of $u_l$ for all $l \in \N$ we get the convergence stated in \eqref{alimtlinf}.
\end{proof}


\noindent
Now we can conclude pointwise convergence of $u$ and $h$:
\begin{Theorem}
	For a.e. $x \in \Omega$ it holds that $\lim_{t\to\infty} u(x,t) = c$ and $\lim_{t\to\infty} h(x,t) = h^*$, where $c \in \{0,1\}$ for $\alpha >1$ and $c=1$ for $\alpha=1$.
\end{Theorem}

\begin{proof}
	For $h$ and in the case $\alpha =1$ convergence follows directly from Theorem \ref{limtlinf}. We consider the case $\alpha>1$. From \eqref{ulpw} we know that (after switching to a subsequence if necessary) $u_l$ converges to $u$ pointwise a.e. in $\Omega \times (0,\infty)$. Hence, for all $(x,t)$ the sequence $(u_l(x,t))_l$ is Cauchy and we can conclude that for large enough $l$ the sequence $u_l(t)$ converges to the same $c \in \{0,1\}$. Hence, for a.e. $(x,t)$ and large enough $l$ it holds that
	\begin{align*}
		|u(x,t) - c| \leq |u(x,t) - u_l(x,t)| + |u_l(x,t)  -c| \rightarrow 0
	\end{align*}
	for $t\to\infty$ and $l\to\infty$ due to \eqref{ulpw} and Theorem \ref{limtlinf}.
\end{proof}

\section{Pattern formation: a 1D study}\label{sec:pattern}
We want to investigate pattern formation in our model (see \cite{M}). For this aim we adapt some of the assumptions on our functions and parameters:
\begin{itemize}
\item $J \in L^1(\R)$, $J(x) = J(-x)$ for $x \in \R$ and $\int_{\R} J(x) \, dx=1$ and $J \circledast u (x) := \int_{\R} J(x-y)u(y) \, dy$, whereas we drop the condition that $0<\eta<J$;
\item $d \in \R$ constant;
\item there is exactly one $h^* >0$ with $g(1,h^*)= 0$, moreover, for this $h^*$ it holds that $\mu(h^*) >0$, $\partial_u g(1,h^*) \geq 0$ and $\partial_h g(1,h^*)<0$. This means that when the cancer cells are at their carrying capacity (corresponding to an acidity level $h^*$), the production of protons is increasing with the cell mass and decreasing with enhancing proton concentration. Indeed, crowded tumor cells are highly hypoxic, and a too acidic environment leads to quiescence or necrosis, thus reducing proton expression. Moreover, we assume that $\mu'(h^*)<0$, thus the growth rate is decreasing with the proton concentration in the neighborhood of the critical value $h^*$.
\item w.l.o.g. we consider $\Omega = [-a,a]$ for $a \in \R$.
\end{itemize}
\noindent
Hence, we consider the model
\begin{align} \label{IBVPA}
\begin{cases}
u_t = du_{xx} + d(uh_x)_x + \mu(h) \ua (1- J \circledast \ub), & x\in \Omega, \, t > 0,\\
h_t = D_H h_{xx} + g(u,h), & x \in \Omega, \, t > 0,\\
u_x = h_x = 0, & x \in \partial \Omega, \, t>0,\\
u(x,0) = u_0(x), \, h(x,0) = h_0(x), & x \in \Omega.
\end{cases}
\end{align}

\subsection{Stability in the local model without diffusion and taxis}

We start by establishing the equilibria of the non-spatial local model that corresponds to \eqref{IBVPA}, i.e. 
\begin{align} \label{locmod}
\begin{cases}
\partial_t u &= \mu(h) \ua (1- \ub), \\
\partial_t h &= g(u,h).
\end{cases}
\end{align}
The biologically more interesting one is given by $(u^*,h^*) = (1,h^*)$, where $h^*$ is the unique solution of $g(1,h) = 0$.  The corresponding characteristic equation of the Jacobian in $(1,h^*)$ is given by
\begin{align*}
\lambda^2 + (\beta \mu(h^*)-\partial_hg(1,h^*))\lambda - \beta \mu(h^*) \partial_h g(1,h^*) = 0.
\end{align*}
The corresponding eigenvalues are 
\begin{align*}
\lambda_{1} = -\beta \mu(h^*) \text{ and } \lambda_2 = \partial_h g(1,h^*)
\end{align*}
and both have negative real parts due to the assumption $\partial_h g (1,h^*) < 0$. Hence, the steady state $(1,h^*)$ is stable in this case.

\subsection{Stability in the local model with diffusion and taxis}
We continue by adding again the diffusion and taxis term to the local model \eqref{locmod}. Adapting the ansatz from \cite{PGM} we consider perturbations of $(1,h^*)$ of the form $u = 1 + \varepsilon \bar{u}(k)$ and $h = h^* + \varepsilon \bar{h}(k)$, where $\bar{u}(k) = \tilde{u} e^{\lambda(k)t}\cos(kx)$ and $\bar{h}(k) = \tilde{h} e^{\lambda(k)t}\cos(kx)$ for $\tilde{u}, \tilde{h} \in \R$ and wavenumber $k \in \N$ and $|\varepsilon| << 1$. Here, $\lambda(k)$ denotes some eigenvalue of the corresponding characteristic equation. As in \cite{N} we use the fact that $\frac{e^{ikx}+e^{-ikx}}{2} = \cos(kx)$ to ensure that our perturbations are real.

\noindent
Inserting these $u$ and $h$ into our model and linearizing about the steady state $(1,h^*)$, we obtain
\begin{align} \label{Linwdt}
\begin{cases}
\lambda(k)\bar{u} &= -k^2d\bar{u} - dk^2\bar{h} - \beta \mu (h^*) \bar{u},\\
\lambda(k)\bar{h} &= -D_Hk^2 \bar{h} + \partial_u g(1,h^*) \bar{u} + \partial_h g(1,h^*)\bar{h}.
\end{cases}
\end{align} 
The corresponding eigenvalues are given by
\begin{align*}
\lambda_{1,2}(k) = \frac{\tr(J_{u,h}(k)) \pm \sqrt{\tr(J_{u,h}(k))^2 -4 \det(J_{u,h}(k))}}{2},
\end{align*}
where we denote by $J_{u,h}$ the Jacobian of the right hand side in system \eqref{Linwdt} at $(1,h*)$ and its determinant and trace are, respectively, given by
\begin{align*}
\tr(J_{u,h}(k))=& -(d+D_H)k^2-\beta\mu(h^*) + \partial_h g(1,h^*) <0,\\
\det(J_{u,h}(k))=& dD_H k^4 + (d (\partial_u g(1,h^*) - \partial_h g(1,h^*)) + \beta \mu(h^*)D_H)k^2  \\
&-\beta \mu(h^*)\partial_hg(1,h^*)>0.
\end{align*}
Hence, the equilibrium $(1,h^*)$ is stable. The local model does not lead to any Turing type patterns.

\subsection{Stability in the nonlocal model}
We consider $u$ and $h$ as in the previous section and linearize the convolution term about $(1,h^*)$ similarly to \cite{PGM}. Hence, inserting $u$ in the convolution term and using the symmetry of $J$, we compute that
\begin{align*}
J\ast \ub  \approx 1+ \varepsilon \beta  \bar{u} (2\pi)^{\frac{1}{2}}\hat{J}(k).
\end{align*}
Here, $\hat{J}$ denotes the Fourier transform of $J$. Hence, linearizing system \eqref{IBVPA}, we obtain
\begin{align} \label{LinIBVPA}
\begin{cases}
\lambda(k)\bar{u} &= -dk^2\bar{u} - dk^2\bar{h} - \beta \mu (h^*)(2\pi)^{\frac{1}{2}}\hat{J}(k) \bar{u},\\
\lambda(k)\bar{h} &= -D_Hk^2 \bar{h} + \partial_u g(1,h^*) \bar{u} + \partial_h g(1,h^*)\bar{h}.
\end{cases}
\end{align}
The corresponding eigenvalues are as above given by
\begin{align*}
\lambda_{1,2}(k) = \frac{\tr(J_{u,h}(k)) \pm \sqrt{\tr(J_{u,h}(k))^2 -4 \det(J_{u,h}(k))}}{2},
\end{align*}
where we denote by $J_{u,h}$ the Jacobian of the right hand side in \eqref{LinIBVPA} at $(1,h^*)$ and its trace and determinant are given by
\begin{align*}
\tr(J_{u,h})(k)=& -(d+D_H)k^2-\beta\mu(h^*)(2\pi)^{\frac{1}{2}}\hat{J}(k) + \partial_h g(1,h^*),\\
\det(J_{u,h})(k)=& dD_H k^4 + (d (\partial_u g(1,h^*) - \partial_h g(1,h^*)) + \beta \mu(h^*)(2\pi)^{\frac{1}{2}}\hat{J}(k)D_H)k^2\\  
&-\beta \mu(h^*)(2\pi)^{\frac{1}{2}}\hat{J}(k)\partial_hg(1,h^*).
\end{align*}
The sign of the real part of the eigenvalues is ambiguous here and depends especially on the sign of $\hat{J}(k)$, which depends on $k$. 
As above, we have stability here if 
\begin{align} \label{stabcond}
\tr{J_{u,h}(k)}<0 \text{ and } \det{J_{u,h}(k)} >0
\end{align}
for all $k = \frac{\pi}{a} z$, where $z \in \mathbb{Z}$. We make this restriction due to our boundary condition $u_x = h_x =0$.\\[-2ex] 

\noindent
Now, we are looking for a critical $k_c$ (that is not necessarily of the form $\frac{\pi}{a} z$) depending on our choice of parameters, where we  distinguish as in \cite{N} the occurrence of Turing instabilities in the case $\text{Im}(\lambda(k_c)) = 0$ for some arbitrary critical $k_c$, Hopf instabilities in the case $\text{Im}(\lambda(0)) \neq 0$, and wave instabilities in the case $\text{Im}(\lambda(k_c))\neq 0$ for some critical $k_c\neq 0$. If $\hat{J}$ is symmetric it suffices to consider only positive $k_c$.\\[-2ex]

\noindent
A Turing bifurcation can occur if we find $k_c$ such that
\begin{align*}
\det(J_{u,h})(k_c)= 0 \text{ and } \tr(J_{u,h})(k_c)<0.
\end{align*}
Now, rewriting these conditions we conclude that the equality
\begin{align}\label{equality}
\hat{J}(k_c)= - d\frac{k_c^2}{\beta(2\pi)^{\frac{1}{2}}\mu(h^*)} \left(1 + \frac{\partial_u g(1,h^*)}{D_Hk_c^2-\partial_hg(1,h^*)}\right)
\end{align}
and the inequality
\begin{align}\label{inequality}
\frac{\partial_h g(1,h^*)-(d+D_H)k_c^2}{\beta \mu  (h^*) (2\pi)^{\frac{1}{2}}} < \hat{J}(k_c)
\end{align}
have to hold for one or several critical $k_c$ in a set $K_c$, whereas \eqref{stabcond} holds for all $k \notin K_c$ that are of the form $\frac{\pi}{n}z$. Such $k_c$ exist depending on the choice of parameters, on the functions $\mu$ and $g$, and especially on the sign of the Fourier transform of $J$. Moreover, due to our assumptions the terms on the right-hand side of \eqref{equality} and on the left-hand side of \eqref{inequality}  are negative and tend to $- \infty$ for $k\rightarrow \pm\infty$.
\noindent
On the other hand, a Hopf or a wave instability can occur if we find $k_c$ such that
\begin{align*}
\tr(J_{u,h})(k_c) = 0  \text{ and } \det{J_{u,h}}(k_c) >0,
\end{align*} 
whereas \eqref{stabcond} holds for all $k$ that do not satisfy this and are of the form $\frac{\pi}{n}z$. Hence, a Hopf instability occurs if
\begin{align*}
\frac{\partial_hg(1,h^*)}{\mu { (h^*)} \beta (2\pi)^{\frac{1}{2}}} = \hat{J}(0) \text{ and } \hat{J}(0) >0,
\end{align*}
whereas \eqref{stabcond} holds for all $k \neq 0$. On the other hand, a wave instability occurs if 
\begin{align*} 
\frac{\partial_h g(1,h^*) - (d+D_H)k_c^2}{\mu (h^*)\beta (2\pi)^{\frac{1}{2}}} = \hat{J}(k_c)
\end{align*}
and 
\begin{align*}
-d\frac{k_c^2}{\beta(2\pi)^{\frac{1}{2}}\mu(h^*)} \left(1 + \frac{\partial_u g(1,h^*)}{D_Hk_c^2-\partial_hg(1,h^*)}\right) < \hat{J}(k_c)
\end{align*}
holds for one or several $k_c \neq 0$, whereas \eqref{stabcond} holds for all other $k$ that are of the form $\frac{\pi}{n}z$ and do not satisfy the above equality and inequality.\\[-2ex]

\noindent
From the above considerations we conclude that the occurrence of a Turing, Hopf or wave instability depends on the concrete choice of $J$, as we need to find suitable $k$ of the form $\frac{\pi}{a} z$. If the Fourier transform $\hat{J}$ is nonnegative, no Turing patterns occur.

\begin{Remark}\label{rem:immer-instabil}
If there is a steady state of the form $(0,h^{**})$ for some $h^{**}>0$ and $\partial_h g(0,h^{**}) \leq 0$, then this equilibrium is stable in the case with diffusion, taxis and nonlocal term. If, on the other hand, $\partial_h g(0,h^{**})>0$, this steady state is unstable already in the case without diffusion and taxis. This case is, however, unrealistic for the biological problem investigated here. Indeed, the proton expression by hypoxic cells is much reduced and there must be at least some very weak acid buffering, lest all cells (and surrounding tissue) become apoptotic.\\[-2ex]

\noindent
Likewise, the steady state $(1,h^*)$ is unstable already in the case without diffusion and taxis if $\partial_h g(1,h^*) >0$. This situation may occur at least in a transient manner, e.g. when the cells can still extrude protons while their environment is quite acidic and if the cells are at their carrying capacity and the proton buffering is relatively low. That can lead, e.g., to a choice of the form $g(u,h)=u+uh-\gamma h^2$ with $\gamma \le 4/5$.
\end{Remark}
\noindent

\section{Numerical simulations}\label{sec:numerics}

In this section we perform numerical simulations of system \eqref{IBVPA}, in order to illustrate the solution behavior. The equations are discretized by using the algorithm in \cite{Nadin2011}; the motility terms were discretized with finite differences (centered for the diffusion, upwind for the drift). The initial conditions are as in \cite{LCS}:
\begin{align*}
u_0(x)=\left \{\begin{array}{cc}
e^{-(x-x_l)^2},&\text{for }x_l<x\le 0\\
e^{-x_l^2}(1-\frac{x}{x_r}),&\text{for }0<x\le x_r\\
\end{array}\right .,\qquad \text{with }x_l=-5,\ x_r=5.
\end{align*}
\noindent
Unless otherwise stated we take $g(u,h)=u(1-h)$, $\mu (h)=\mu /(1+h)$, with $\mu>0$ a constant and $d=1$.

\noindent
In a first test we took $\beta=1$, $\mu =1$, along with the logistic kernel $J(x)=1/(2+e^x+e^{-x})$ (see, e.g., \cite{Lee2009}) and the uniform kernel $J(x;\rho)=\frac{1}{2\rho}\mathds 1_{[-\rho,\rho]}$. The first two columns of Figure \ref{fig:1} show simulation results for $\alpha =2$, which is the 'limit value' in \eqref{cond:alphabeta}. 
The solution ceased (in finite time) to exist for sufficiently large $\alpha $ in each of these situations ($\alpha \sim 6.25$ and $\alpha \sim 8.2$, respectively), $u$ exhibiting strong aggregation near the initial bulk of cells, cf. last two columns in Figure \ref{fig:1}. 
This behavior was also observed for increasing values of $\mu$, with the difference of singularities already occuring for smaller $\alpha $ values.\\[-2ex] 

\begin{figure}[h!]
	\begin{minipage}[h]{.24\linewidth}
		\raisebox{1.2cm}
			{\includegraphics[width=1\linewidth, height = 2.95cm]{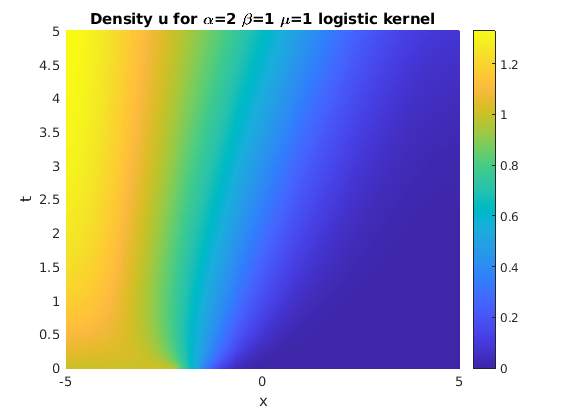}}\\[-4ex]
	\raisebox{2.3cm}
	{\includegraphics[width=1\linewidth, height = 2.95cm]{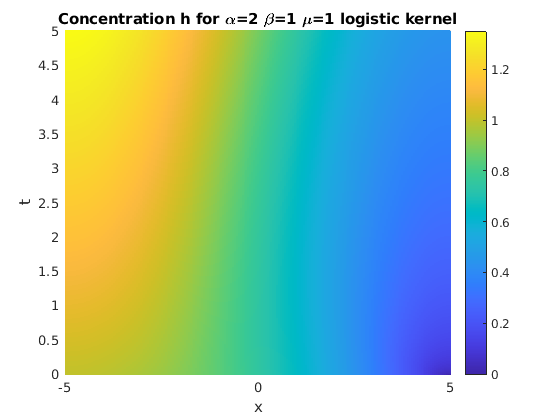}}
\end{minipage}%
	\hspace{0.01cm}
		\begin{minipage}[h]{.24\linewidth}
		\raisebox{1.2cm}
		{\includegraphics[width=1\linewidth, height = 2.95cm]{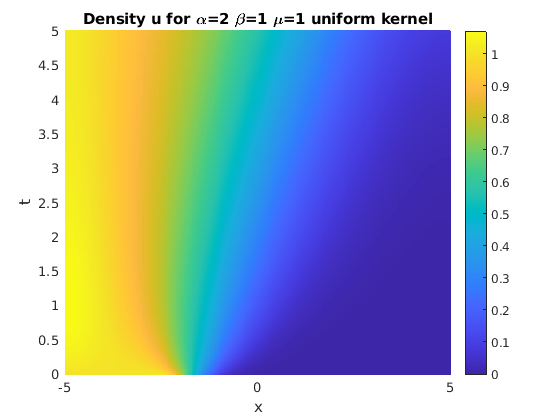}}\\[-4ex]
	\raisebox{2.3cm}
		{\includegraphics[width=1\linewidth, height = 2.95cm]{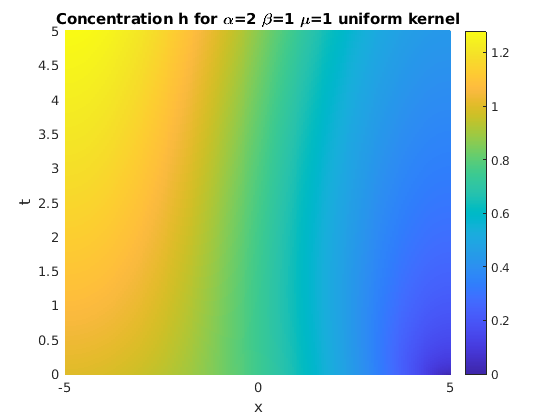}}
\end{minipage}%
	\begin{minipage}[h]{.24\linewidth}	
\raisebox{1.2cm}
		{\includegraphics[width=1\linewidth, height = 2.95cm]{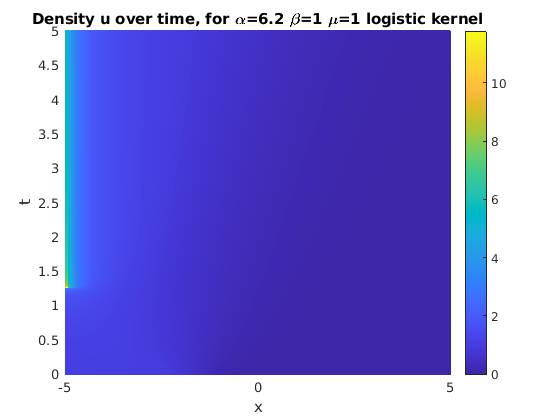}}\\[-4ex]
\raisebox{2.3cm}
{\includegraphics[width=1\linewidth, height = 2.95cm]{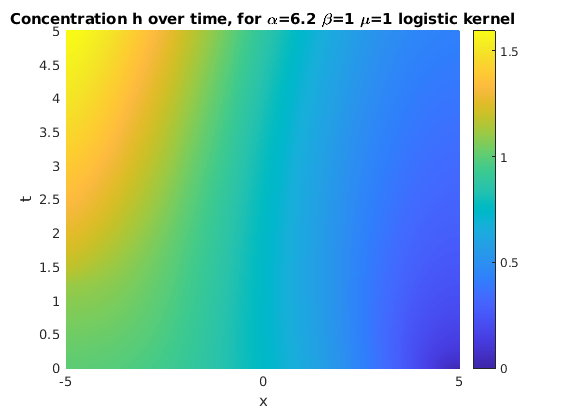}}		
\end{minipage}%
\hspace{0.01cm}		
			\begin{minipage}[h]{.24\linewidth}	
\raisebox{1.2cm}
	{\includegraphics[width=1\linewidth, height = 2.95cm]{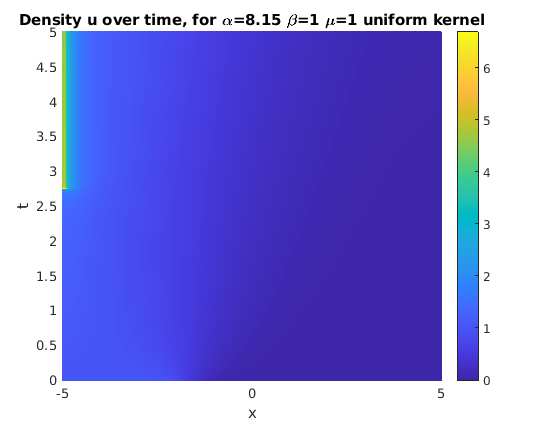}}\\[-4ex]
	\raisebox{2.3cm}
	{\includegraphics[width=1\linewidth, height = 2.95cm]{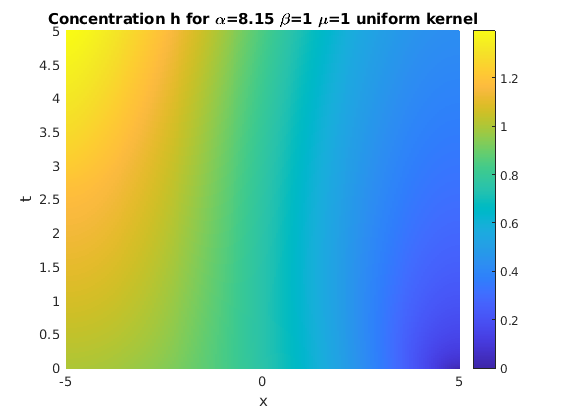}}
	\end{minipage}%
	\vspace*{-2cm}
	\caption{Simulation results for \eqref{IBVPA} with $\beta=\mu=1$. First two columns: $\alpha =2$, 3rd column: $\alpha =6.2$, last column: $\alpha =8.15$. Uniform kernel used with $\rho=1$.}\label{fig:1}
\end{figure}

\noindent
Increasing the values of $\mu$ and $\beta $ leads to patterns, the shape of which depends decisively on the interaction kernel $J$ and also on the values of $\alpha $ and $d$. Figure \ref{fig:2} shows 1D space-time patterns of the cell density $u$ for $\beta=20$, $\mu=100$, and several combinations of $\alpha$ and $J$. The results for the proton concentration $h$ are not shown, as there are only small quantitative differences between the respective cases. Figure \ref{fig:2} suggests that, irrespective of the chosen kernel\footnote{We performed simulations with several other kernels, including the so-called 'Mexican hat' (also known as Ricker wavelet, see e.g. \cite{Ei,Zaytseva2019} for its use in related, but different contexts), cosine, and Epanechnikov.}, higher cooperative intraspecific interactions (larger $\alpha$ values) or slower diffusion delay the invasion of cells in the whole region, leading instead to enhanced proliferation. On the long run the cells tend to fill the whole space and remain at their carrying capacity. This behavior endorses the results in Section \ref{sec:asymptotic} and is  particularly well visible for the logistic kernel, which satisfies all conditions in the proofs of the theoretical results of Sections \ref{sec:analysis} and \ref{sec:asymptotic}; the process is much slower when a uniform kernel is used, however it has eventually the same outcome. The last row in Figure \ref{fig:2} exhibits the situation of a cell diffusion which is much slower than that of protons. The effect is a delayed filling of the space with cells (and produced protons) and a later formation of the patterns observed in the upper rows. The asymptotic behavior is similar, only it takes longer for the solution to reach the respective states. \\[-2ex]

\begin{figure}[h!]
	\begin{minipage}[h]{.24\linewidth}
		\raisebox{1.2cm}
		{\includegraphics[width=1\linewidth, height =2.95cm]{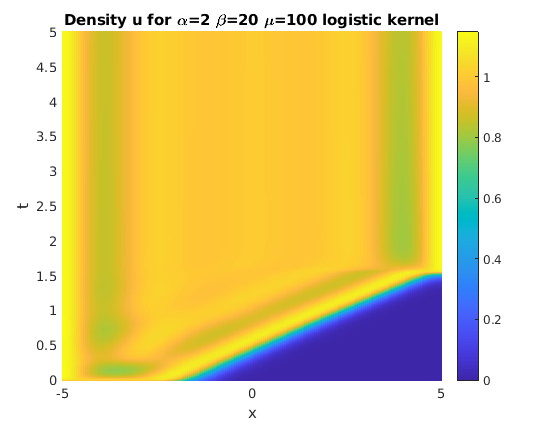}}\\[-4ex]
		\raisebox{2.3cm}
		{\includegraphics[width=1\linewidth, height = 2.95cm]{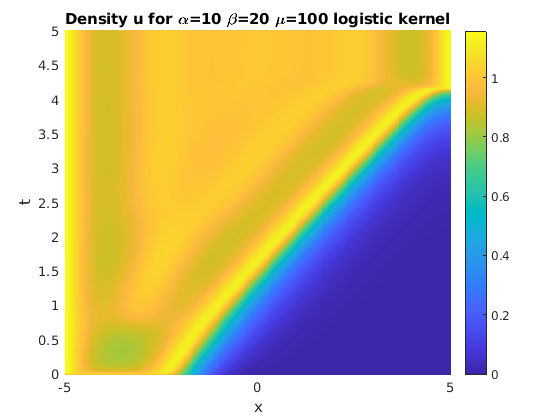}}\\[-10ex]
		\raisebox{2.3cm}
		{\includegraphics[width=1\linewidth, height = 2.95cm]{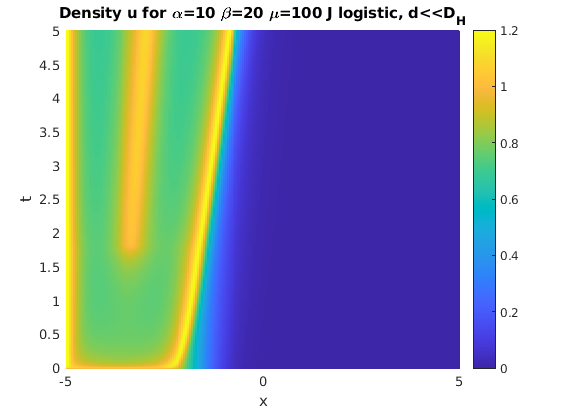}}
	\end{minipage}%
	\hspace{0.01cm}
	\begin{minipage}[h]{.24\linewidth}
		\raisebox{1.2cm}
		{\includegraphics[width=1\linewidth, height =2.95cm]{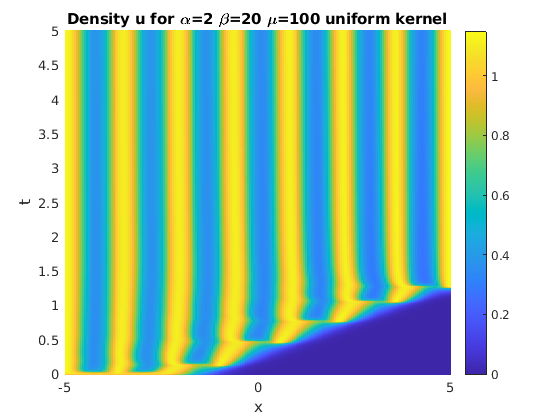}}\\[-4ex]
		\raisebox{2.3cm}
		{\includegraphics[width=1\linewidth, height = 2.95cm]{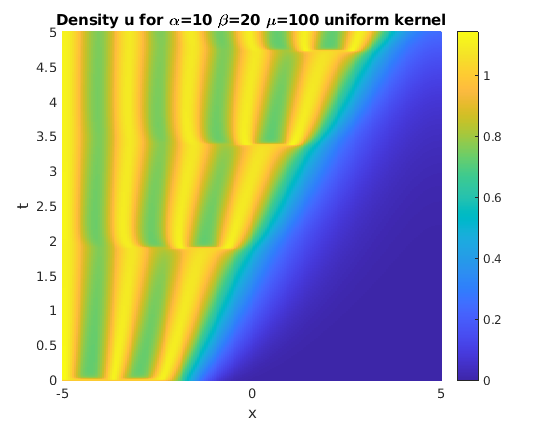}}\\[-10ex]
		\raisebox{2.3cm}
		{\includegraphics[width=1\linewidth, height = 2.95cm]{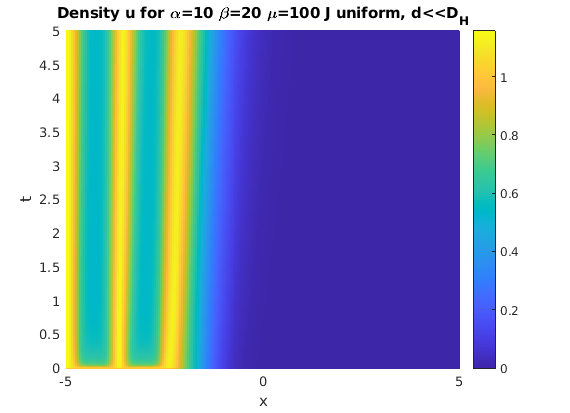}}
	\end{minipage}%
	\begin{minipage}[h]{.24\linewidth}	
		\raisebox{1.2cm}
		{\includegraphics[width=1\linewidth, height =2.95cm]{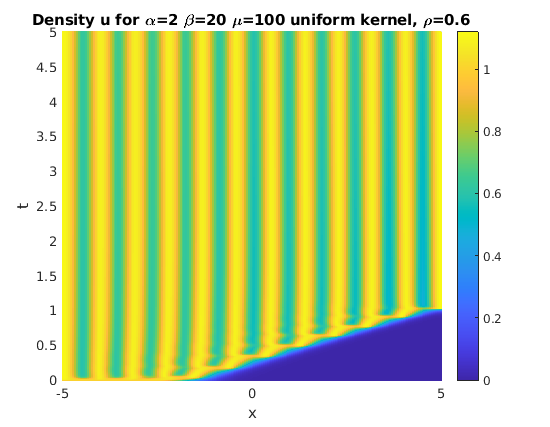}}\\[-4ex]
	\raisebox{2.3cm}
	{\includegraphics[width=1\linewidth, height = 2.95cm]{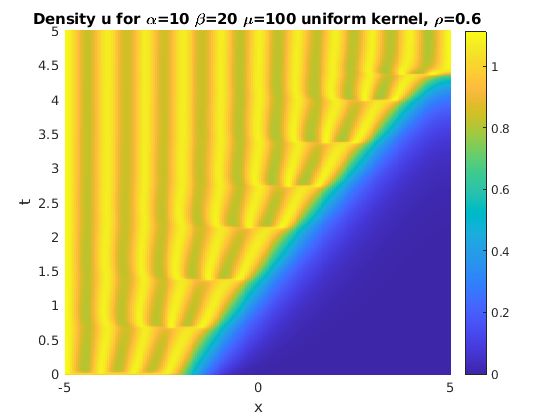}}	\\[-10ex]	
			\raisebox{2.3cm}
		{\includegraphics[width=1\linewidth, height = 2.95cm]{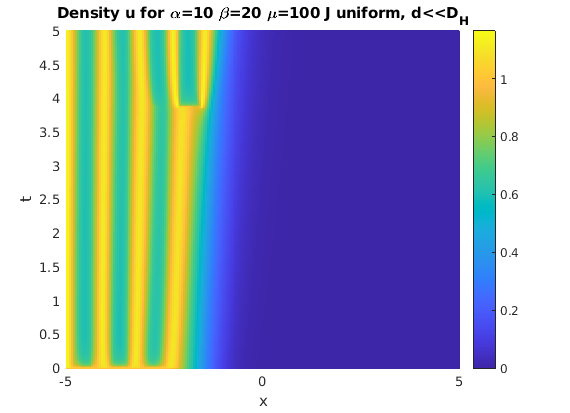}}		
	\end{minipage}
	\hspace{0.01cm}		
	\begin{minipage}[h]{.24\linewidth}	
		\raisebox{1.2cm}
		{\includegraphics[width=1\linewidth, height =2.95cm]{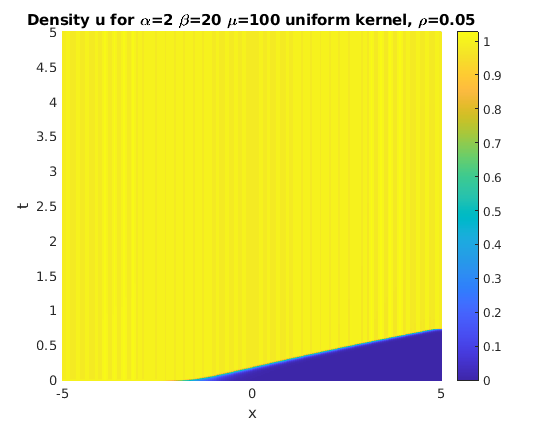}}\\[-4ex]
		\raisebox{2.3cm}
		{\includegraphics[width=1\linewidth, height = 2.95cm]{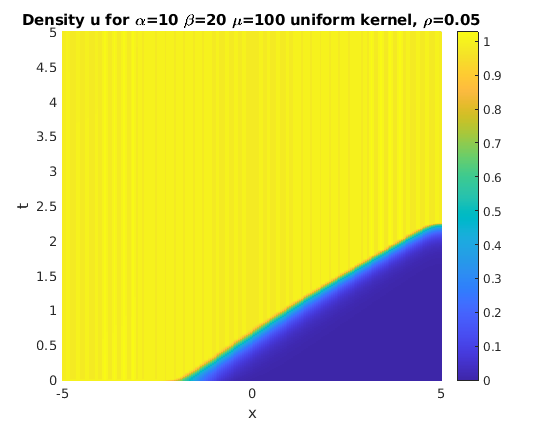}}\\[-10ex]
		\raisebox{2.3cm}
		{\includegraphics[width=1\linewidth, height = 2.95cm]{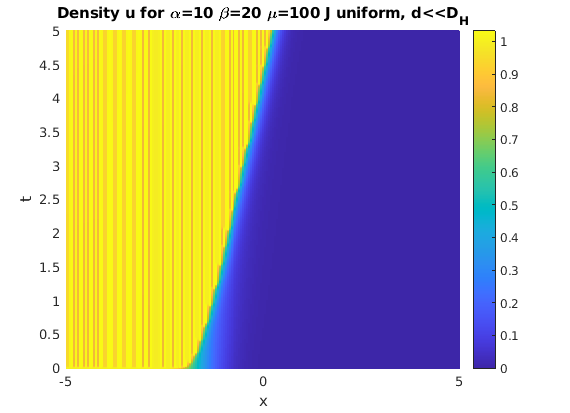}}
	\end{minipage}
	\vspace*{-2cm}
	\caption{Simulation results for \eqref{IBVPA} with $\beta =20$ and $\mu=100$. Upper row: $\alpha =2$, lower row: $\alpha =10$. First  column $J$ logistic, other columns $J$ uniform: 2nd column: $\rho=1$, 3rd column: $\rho=0.6$, 4th column: $\rho=0.05$. Upper rows: $d=D_H=1$, last row: $d\ll D_H$.}\label{fig:2}
\end{figure}

\noindent
To assess the effect of nonlocality  we performed simulations with the source term in the $u$-equation of \eqref{IBVPA} replaced by $\mu(h)u^\alpha(1-u^\beta)$. The results are shown in Figure \ref{fig:3}. The first two columns illustrate the case with the same source term for proton concentration as above, namely $g(u,h)=u(1-h)$, for which no patterns seem to develop (we tried several combinations of parameters, including those used for the patterns in Figure \ref{fig:2}). In fact, decreasing the value of $\rho$ in the uniform kernel $J(x;\rho)$ eventually leads to the local version of the system. 
The plots in the leftmost column were produced with $d=D_H$, while those in the middle column used $d\ll D_H$. The behavior of $u$ and $h$ is the same, with the difference of the second case inferring a slower spread of cells and protons. The last column in Figure \ref{fig:2} already shows the tendency of disappearing patterns when approaching the local case. The last column of Figure \ref{fig:3} shows the case where the source term in the $h$-equation is replaced by $g(u,h)=u+uh-\gamma h^2$, as proposed in Remark \ref{rem:immer-instabil}.\footnote{We tried several other source terms satisfying the conditions in Remark \ref{rem:immer-instabil}, e.g. $g(u,h)=uh/(1+uh+h)$, all resulting in the same qualitative behavior.}\\[-2ex]

\noindent
 No patterns for $u$ were observed for the local model, which, together with the simulations performed for intermediary values of $\rho$, suggests that the patterns are driven by the nonlocality of cell-cell interactions, more precisely by intraspecific competition. The simulations also confirm the long time behavior of the system, even in the local case.
 

\begin{figure}[h!]
\begin{center}	
	\begin{minipage}[h]{.28\linewidth}
		\raisebox{1.2cm}
		{\includegraphics[width=1\linewidth, height =2.95cm]{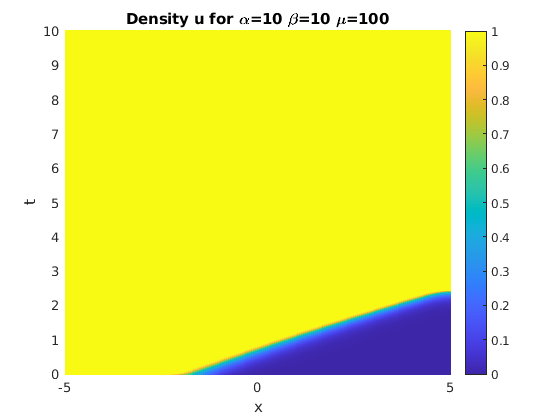}}\\[-4ex]
		\raisebox{2.3cm}
		{\includegraphics[width=1\linewidth, height = 2.95cm]{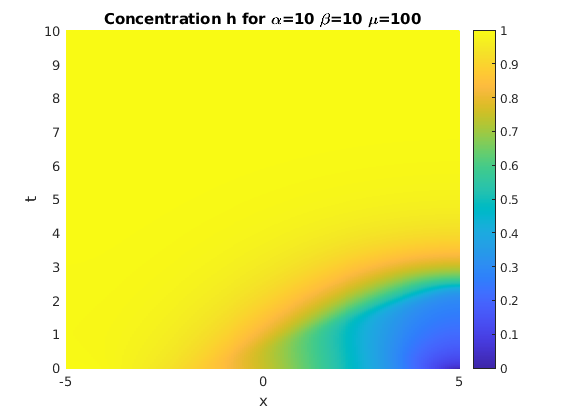}}
	\end{minipage}%
	\hspace{0.01cm}
	\begin{minipage}[h]{.28\linewidth}
		\raisebox{1.2cm}
		{\includegraphics[width=1\linewidth, height =2.95cm]{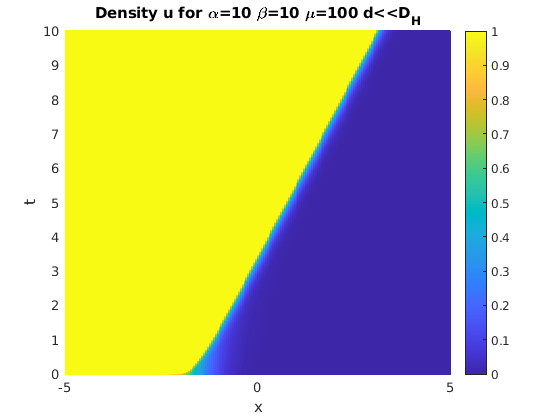}}\\[-4ex]
		\raisebox{2.3cm}
		{\includegraphics[width=1\linewidth, height = 2.95cm]{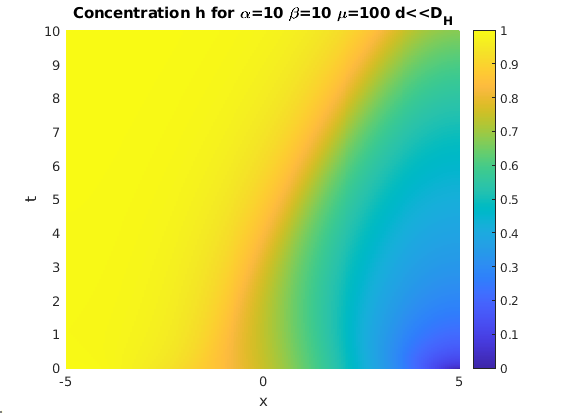}}
	\end{minipage}%
	\begin{minipage}[h]{.28\linewidth}	
		\raisebox{1.2cm}
		{\includegraphics[width=1\linewidth, height =2.95cm]{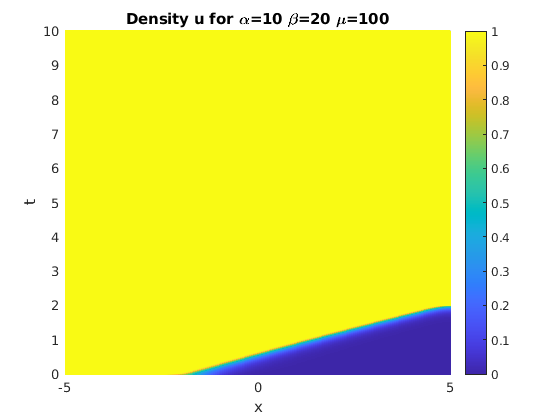}}\\[-4ex]
		\raisebox{2.3cm}
		{\includegraphics[width=1\linewidth, height = 2.95cm]{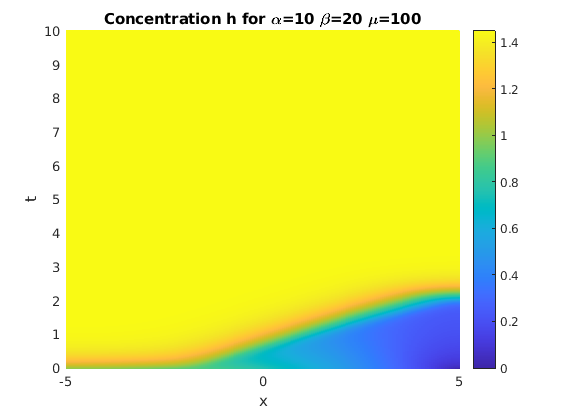}}		
	\end{minipage}
\end{center}
	\vspace*{-2cm}
	\caption{Simulation results for \eqref{IBVPA} with local source term $\mu(h)u^\alpha(1-u^\beta )$ replacing the one in the equation for $u$. Left and middle column: $g(u,h)=u(1-h)$ with $d=D_H$ and $d\ll D_H$, respectively. Right column: $g(u,h)=u+uh-\gamma h^2$, $d=D_H$.}\label{fig:3}
\end{figure}

\section{Discussion}\label{sec:discussion}

In this note we investigated a model describing pH-tactic behavior of cells with nonlocal source terms. As such, this work is extending the one in \cite{LCS}, which studied the Fisher-KPP equation with nonlocal intraspecific competition with various powers of the solution. In contrast to \cite{LCS} we handled here a problem in a bounded domain, and the population dynamics was coupled to that of the proton concentration, which also led to a taxis term. The proof of our results concerning global well-posedness and long time behavior relied, however, to a substantial  extent on the methods in \cite{LCS}. We also dealt here with space-dependent tensor coefficients in the motility terms, which involve myopic rather than Fickian diffusion. The dissipative effect of the repellent pH-taxis contributed to reducing some of the difficulties in the analysis - as long as the required conditions on the functions involved in the system are satisfied. \\[-2ex]

\noindent
Among the relatively few existing models with nonlocal source terms, the one in \cite{SMLC} is closely related, however it features several differences: the cells perform attractive haptotaxis towards gradients of extracellular matrix (ECM), the nonlocal source terms are contained in both equations, do not involve any powers, and the Fickian diffusion of cells has a constant coefficient. Our model requires less regularity for the interaction kernel and the motility coefficients involve a tensor and are more general. On the other hand, the nonexploding solution behavior is  favorized in our case by repellent chemotaxis. We also provided an informal model deduction and an assessment of the long time solution behavior. The analysis done in \cite{Negreanu2013} for a model with standard motility and with nonlocal source terms as in \cite{SMLC}, but with one or two species performing chemotaxis towards the same attractant imposes certain requirements on the forcing term of the latter, mainly in order to obtain the asymptotic behavior of the cell-related solution components. Our condition \eqref{assg} imposed for similar purposes on the source term of the tactic signal looks rather differently. The attraction-repulsion chemotaxis models considered in \cite{Ren2022} have closer similarities with our setting, as far as the nonlocal intraspecific interactions are concerned. Major differences occur through our system only featuring two equations, in the source terms of the chemical cues, and in the motility terms: the latter involve in our case the space-dependent tensor $\mathbb D(x)$ and myopic diffusion, while the nonlocal reaction term in the proton dynamics is more general. We also prove an explicit long time behavior of both solution components and provide a short analysis of space-time patterns (in 1D), along with numerical simulations.\\[-2ex]

\noindent
Our preliminary analysis in Section \ref{sec:pattern} and the simulation results in Section \ref{sec:numerics} suggest that patterns occur only in the nonlocal model, are not of Turing type, and seem to be driven by the nonlocal source terms and influenced by the chosen kernel and the combination of parameters in the nonlocal term. This is in line with the pattern behavior observed in \cite{LCS} and with other works concerning reaction-diffusion problems with nonlocal intra- and/or interspecific competition, cf. e.g. \cite{Fuentes2004,Han2019,PGM,Segal2013,Tian2017,Zaytseva2019}. Those works involved more or less similar source terms and no taxis, however the repellent pH-taxis contained in our model does not seem to have a relevant influence on the patterns. \\[-2ex]

\noindent
Open problems relate to a thorough study of patterns depending on the interplay between the parameters $\alpha$, $\beta$, $\mu$ and the influence of the kernel $J$. Moreover, the well-posedness, asymptotic and blow-up behavior, along with patterning are largely unknown in the case of a  degenerating motility tensor - the less so in combination with myopic diffusion and/or other types of taxis. Indeed, these can lead in the local case to very complex issues even in 1D, as shown e.g. in \cite{winkler,WiSu}.

\phantomsection
\printbibliography

\end{document}